\documentclass{article}
\usepackage{amsmath}
\usepackage{bbm}
\usepackage{amsthm}
\usepackage{amssymb}
\usepackage{hyperref}
\usepackage{verbatim}
\usepackage{graphicx}
\usepackage[usenames,dvipsnames,svgnames,table]{xcolor}
\usepackage{mathrsfs}  
\usepackage{accents} 
\usepackage{enumerate}

\usepackage{mathtools}
\usepackage[T1]{fontenc}
\def\D{\text{\DH}}
\usepackage{graphicx}

\usepackage[hmargin=3cm,vmargin=3.5cm]{geometry}
\def\a{\alpha}
\def\b{\beta}
\def\ga{\gamma}

\def\de{\delta}
\def\De{\Delta}
\def\ep{\epsilon}

\def\la{\lambda}
\def\La{\Lambda}
\def\si{\sigma}

\def\th{\theta}
\newcommand{\Th}[0]{\Theta}

\newcommand{\vphi}[0]{\varphi}

\def\DD{{\cal D}}

\def\LL{{\cal L}}

\def\RR{{\cal R}}
\def\SS{{\cal S}}

\newcommand{\N}[0]{\mathbb{N}}
\newcommand{\F}[0]{\mathbb{F}}
\newcommand{\R}[0]{\mathbb{R}}
\newcommand{\Z}[0]{\mathbb{Z}}

\newcommand{\C}[0]{\mathbb{C}}

\newcommand{\T}[0]{\mathbb{T}}

\newcommand{\ost}[1]{\accentset{\ast}{#1}}
\newcommand{\dmd}[1]{\accentset{\diamond}{#1}}
\newcommand{\ed}{\dmd{e}_u^{1/2}}

\newcommand{\supp}{\mathrm{supp} \,}
\newcommand{\suppt}{\mathrm{supp}_t \,}

\newcommand{\fr}[2]{\frac{#1}{#2}}

\newcommand{\ALI}[1]{\begin{align*} #1 \end{align*}}
\newcommand{\ali}[1]{\begin{align} #1 \end{align}}

\newcommand{\leqc}[0]{\lesssim}

\newcommand{\pr}[0]{\partial}
\newcommand{\nb}{\nabla}

\newcommand{\co}[1]{\left \|#1 \right \|_{C^0}}
\newcommand{\ca}[1]{\left \|#1 \right \|_{C^\a_x}}

\newcommand{\va}[0]{\vec{a}}

\newcommand{\vb}[0]{\vec{b}}

\DeclareMathAlphabet{\mathpzc}{OT1}{pzc}{m}{it}

\newcommand{\hc}[0]{\widehat{C}}

\newcommand{\eu}[0]{e_u^{1/2}}

\newcommand{\nat}[0]{(\Xi \eu)}

\def\XXint#1#2#3{{\setbox0=\hbox{$#1{#2#3}{\int}$}
     \vcenter{\hbox{$#2#3$}}\kern-.5\wd0}}

\newtheorem{thm}{Theorem}
\newtheorem{lem}{Lemma}[section]

\newtheorem{prop}{Proposition}[section]
\newtheorem{cor}{Corollary}[section]

\newtheorem{remk}{Remark}[section]

\theoremstyle{definition}
\newtheorem{defn}{Definition}[section]

\theoremstyle{remark}

\title{A direct approach to convex integration and failure of compactness for the SQG equation}
\author{Philip Isett and Andrew Ma}
\date{ }

\begin{document}
\begin{center} 
{\Large A direct approach to nonuniqueness and failure of compactness for the SQG equation}

$ $

Philip Isett\footnote{Department of Mathematics, University of Texas at Austin, Austin, TX 78712, USA \href{mailto:isett@math.utexas.edu}{isett@math.utexas.edu} } and Andrew Ma\footnote{Department of Mathematics, University of Texas at Austin, Austin, TX 78712, USA  \href{mailto:andyma@math.utexas.edu}{andyma@math.utexas.edu} } 
\end{center}

\begin{abstract}  We give an alternative proof of the nonuniqueness of weak solutions to the surface quasigeostrophic  equation (SQG) first shown in \cite{buckShkVicSQG}.  Our approach proceeds directly at the level of the scalar field.  Furthermore, we prove that every smooth scalar field with compact support that conserves the integral can be realized as a weak limit of solutions to SQG.%
\end{abstract}


\section{Introduction}
In this work, we are concerned with the two-dimensional surface quasi-geostrophic (SQG) equation, which is the following transport equation for a scalar field $\th : \R \times \T^2 \to \R$ 
\ali{
\label{eq:SQG}
\pr_t \th + \nb \cdot ( \th u) &= 0,  \qquad  
u \coloneqq \nb^\perp \La^{-1} \th, 
}
where $\La = (-\De)^{1/2}$, and $\nb^\perp = (-\pr_2, \pr_1)$.  

This equation is a basic example of an {\it active scalar} equation, so called because the drift velocity $u$ depends at every time (nonlocally) on the scalar field $\th$ that is being transported.  The SQG equation arises as a model in geophysical fluid dynamics, where it has applications to both meteorological and oceanic flows \cite{held1995surface, pedlosky2013geophysical}.  In this context, the field $\th$ represents temperature or surface buoyancy in a certain regime of stratified fluid flow.  The equation has been studied extensively in the mathematical literature due to its close analogy with the 3D incompressible Euler equations and the problem of blowup for initially classical solutions, which remains open as it does for the Euler equations.  A survey of developments is given in the introduction to \cite{buckShkVicSQG}.

Fundamental to the study of the SQG equation are the following basic conservation laws:
\begin{enumerate}[i]
\item For all sufficiently smooth solutions, the {\it Hamiltonian} $\fr{1}{2} \int_{\T^2} |\La^{-1/2} \th(x,t)|^2 dx$ remains constant. 
\item For all sufficiently smooth solutions, the $L^p$ norms $\| \th(t) \|_{L^p(\T^2)}$ remain constant $1 \leq p \leq \infty$, as do the integrals $\int_{\T^2} F(\th(x,t)) dx$ for any smooth function $F$. 
\item For all weak solutions to SQG, the integral $\int_{\T^2} \th(x,t) dx$ remains constant.
\end{enumerate}

\noindent(To prove (i), multiply \eqref{eq:SQG} by $\La^{-1} \th$ and integrate by parts.  To prove (ii), use $\nb \cdot u = 0$ to check that $F(\th)$ satisfies $\pr_tF(\th) + \nb\cdot(F(\th) u ) = 0$, then integrate.  To prove (iii), simply integrate in space.)

Note that, in contrast to (iii), the nonlinear laws (i) and (ii) require that the solution is ``sufficiently smooth''.  If one expects that turbulent SQG solutions have a dual energy cascade as in the Batchelor-Kraichnan predictions of 2D turbulence \cite{constScalingScalar,constantin2002energy,buckShkVicSQG}, then one has motivation to consider weak solutions that are not smooth.  A basic open question for the SQG equations is then: What function spaces represent the minimal amount of smoothness required for the conservation laws to hold?  This question is exactly the concern of the (generalized) {\it Onsager conjectures} for the SQG equation.  A closely related problem is to find the minimal regularity that guarantees uniqueness for the initial value problem.  

Using H\"{o}lder spaces as a natural scale to measure spatial regularity, the generalized Onsager conjectures for SQG can be stated as follows.  (The space-time regularity below is on a finite time interval.)   
\begin{enumerate}[i]
\item If $\th \in C^0$, then the conservation of the Hamiltonian 
holds.  However, for any $0 < \a < 1/2$ there exist weak solutions of class $\La^{-1/2} \th \in L_t^\infty C_x^\a$ that fail to conserve the Hamiltonian.
\item Let $F \in C^\infty(\R)$.  Then if $\th \in L_t^\infty C_x^\a$ for some $\a > 1/3$, the law $\int_{\T^2} F(\th(x,t))dx \equiv \mbox{const}$ is satisfied.  On the other hand, for any $\a < 1/3$, there exist solutions $\th \in L_t^\infty C_x^\a$ that violate this law.
\end{enumerate}

Some remarks about these conjectures are in order:
\begin{enumerate}
\item These conjectures generalize the original Onsager conjecture \cite{onsag}, which concerned turbulent dissipation in the incompressible Euler equations and stated that the H\"{o}lder exponent 1/3 should mark the threshold regularity for conservation of energy for solutions to the incompressible Euler equations.  
See \cite{deLSzeCtsSurv} for discussion of the significance of Onsager's conjecture in turbulence theory.  
\item The conjectured threshold exponents are derived from the fact that the conservation law for sufficiently regular solutions has been proven in both cases (i) and (ii).  Namely, \cite{isettVicol} proves conservation of the Hamiltonian for solutions with $\th \in L^3(I \times \T^2)$, while \cite{akramov2019renormalization} proves the conservation law (ii) for $\a > 1/3$.  The proofs are variants of the kinematic argument of \cite{CET}, which proved energy conservation for the Euler equations above Onsager's conjectured threshold.
\item Following the seminal work \cite{deLSzeCts}, advances in the method of convex integration have made possible the pursuit of Onsager's conjecture both for the Euler equations and more general fluid equations.  In particular, Onsager's conjecture for the 3D Euler equations has been proven in \cite{isettOnsag} (see also \cite{buckDeLSVonsag,isett2017endpoint}), while the first progress towards the Onsager conjecture (i) for SQG has been made in \cite{buckShkVicSQG}.  See \cite{deLOnsagThm,buckVicsurvey} for surveys and \cite{klainerman2017nash} for a discussion of generalized Onsager conjectures.  
\item To make sense of the conjecture in part (i), it must be noted that the SQG equation is well-defined for $\th$ having negative regularity.  Namely, for any smooth vector field $\phi(x)$ on $\T^2$, the quadratic form $\int_{\T^2} \th \nb^\perp \La^{-1}\th \cdot \phi(x) dx$, initially defined for smooth $\th$, has a unique bounded extension to $\th \in H^{-1/2}$.  This fact, which relies on the anti-self-adjointness of the operator $\nb^\perp \La^{-1}$, allows the SQG nonlinearity to be well-defined in $\DD'$ for $\th$ of class $\th \in L_t^2 H_x^{-1/2}$ (see \cite[Definition 1.1]{buckShkVicSQG}).
\end{enumerate}
The boundedness of the nonlinearity in a negative Sobolev space is key to constructing weak solutions to the SQG equations by compactness methods \cite{resnick1995}.  These solutions obtained in \cite{resnick1995, marchand2008existence} are known to exist for all time and to have $L^p$ norms bounded uniformly in $t$ by the initial data, but it remains unknown whether they are uniquely determined by their initial data.

The \cite{resnick1995} construction of weak solutions is closely tied to the phenomenon of {\bf weak compactness} for SQG solutions.  Namely weak limits of sequences of solutions to SQG in the space $L^\infty$ weak-* must also be solutions to the SQG equations.  The proof (see e.g. \cite{isettVicol}) relies on the boundedness of the nonlinearity in a negative Sobolev space.  The significance of weak compactness in the present context is that weak compactness appears to present an obstruction to attacking Onsager's conjecture and to proving nonuniqueness using convex integration methods when it is present \cite{deLSzeHFluid}.



The main goals of the present paper are to provide an alternative approach to the nonuniqueness for SQG first shown in \cite{buckShkVicSQG}, thus offering a different point of view from which to pursue the Onsager conjecture for SQG, and to establish an ``h-principle'' result that demonstrates the stark failure of weak compactness for SQG weak solutions.  Our main results are the following.

\begin{thm}[Existence of Weak Solutions] \label{thm:existence}
	For any $0 < \a < 3/10$ and $0 < \b <1/4$ there exists nontrivial weak solutions $\th$ to SQG with compact support in time such that the potential function $\La^{-1/2}\th$ is in the H\"{o}lder class $C_t^0C_x^\a \cap C_t^\b C_x^0 (\R \times \T^2)$. 
\end{thm}
\begin{thm}[h-principle]
\label{thm: h-principle}
	Fix any $\a < 3/10, \b <1/4$, and let $f$ be a smooth scalar field with compact support in time that staisfies the conservation law $\int_{\T^2}f(x,t) dx = 0$ as a function of time.  Then there exists a sequence of solutions to SQG, $\{\th_n\}_{n \in \N}$, such that each scalar field $\th_n$ has compact support in time and a corresponding potential function $\La^{-1/2} \th_n$ of H\"{o}lder class $\La^{-1/2} \th_n \in C_t^\b C_x^0 \cap C_t^0C_x^\a (\R \times \T^2)$, and such that $\La^{-1/2}\th_n \rightharpoonup \La^{-1/2}f$ in the $L^\infty(\R \times \T^2)$ weak-$\ast$ topology. 
\end{thm}

The ``h-principle'' result stated in Theorem~\ref{thm: h-principle} is a new result to this paper, while Theorem~\ref{thm:existence} was first obtained in the work \cite{buckShkVicSQG}.  
(The main result of \cite{buckShkVicSQG} is stated in terms of $\La^{-1} \th$ rather than $\La^{-1/2} \th$, but in terms of spatial regularity the results are equivalent.)  

The main contributions of our work are as follows:
\begin{enumerate}
\item  One of the main ideas of \cite{buckShkVicSQG} to apply convex integration to the SQG equation is to recast the SQG equation in the following {\it momentum form} for the unknown $v = \La^{-1} u = \nb^\perp \De^{-1} \th$:
\ALI{
	\pr_t v + u \cdot \nb v - (\nb v)^T \cdot u = - \nb p, \qquad \mbox{div } v = 0, \qquad u = \La v.
}
The authors are then able to perform convex integration at the level of the vector field $v$ rather than the scalar field $\th$.  The scalar field $\th$ can then be recovered from $v$ by setting $\th = - \nb^\perp \cdot v$.  

In this work, we take a different approach, working directly at the level of the scalar field $\th$.  To execute this approach, we require the error term in the construction to have the structure of a second order divergence, as opposed to the first order divergence form used in \cite{isettVicol}.  We also employ the ``constant trick'' from \cite{corFarGanPor,shvConvInt, isettVicol} that the divergence of a function of $t$ is zero.

Given that the Onsager conjecture for SQG remains open, it is useful to have more than one approach to constructing SQG solutions.  Having a second approach gives a separate angle from which to consider the problem and opens the door to considering more general active scalar equations.  This approach might also be useful for the still open Onsager conjecture for 2D Euler.  Furthermore, it is desirable from a physical point of view to have an approach that works directly at the level of $\th$, since the variable $\th$ has a clear physical meaning.  We also note there is yet a different approach to convex integration for SQG that has been obtained independently in \cite{chengLiKwon}.  This work proceeds at the level of $\La^{-1} \th$ and considers the steady state SQG equation.  

\item  While typically theorems on the failure of compactness can be obtained as essentially a byproduct of a successful convex integration scheme, the argument of \cite{buckShkVicSQG} is not quite structured towards proving an h-principle result.  Our proof of Theorem~\ref{thm: h-principle} involves taking a different organizational strategy in the estimates compared to \cite{buckShkVicSQG}.  The proof of the h-principle also involves an additional approximation step compared to similar results in \cite{isettVicol, IOnonpd}, as the convex integration scheme implemented here is required to have compact frequency support.  
\item  A key idea of \cite{buckShkVicSQG} is to express the nonlinearity $u \cdot \nb v - (\nb v)^T \cdot u$ as the sum of the divergence of a 2-tensor and the gradient of a scalar function to have good control over high-high to low interactions.  Here we show that the desired divergence form can be obtained for the nonlinearity $\th \nb^\perp \La^{-1} \th$ of SQG, and we extend the derivation so that it applies to general odd multipliers.  
\item Finally, we employ a ``bilinear microlocal lemma'' analogous to \cite[Lemma 4.1]{isettVicol}.  This lemma allows us to obtain a better estimate on the low frequency part of the error compared to the corresponding technique in \cite{buckShkVicSQG}, which relies on the use of sharp time cutoffs to approximate nonlinear phase functions by their initial conditions.  Having such an improved estimate is needed for schemes that aim to improve the H\"{o}lder regularity.  The estimate we obtain is compatible with an ideal case scenario for SQG, so that the only terms now limiting the regularity below the conjectured threshold are the high-frequency interference terms.
\end{enumerate}

The statement of Theorem~\ref{thm: h-principle} on the h-principle helps clarify an important point about the method of convex integration, which is that the method appears so far to apply only in cases where an appropriate h-principle implying failure of compactness can be proven \cite{deLSzeHFluid}.  Namely, Theorem~\ref{thm: h-principle} shows that even though weak compactness does hold for SQG in $L^\infty$, which prevents the schemes of \cite{shvConvInt, isettVicol} for active scalar equations from applying to SQG, there is still failure of compactness in the space $\La^{-1/2} \th \in L^\infty$, thus permitting the scheme of \cite{buckShkVicSQG} and the present work to succeed.  The statement of Theorem~\ref{thm: h-principle} is modeled off h-principles proven in \cite{isettVicol, IOnonpd}, which also relate weak limits of solutions to conservation laws.  In general, ``h-principles'' in PDE are modeled off the result of Nash \cite{nashC1} that $C^0$ limits of isometric immersions of closed $n$-manifolds into $\R^{n+2}$ can realize any smooth, short immersion.  For further examples of h-principles in fluids, we refer to \cite{choff, danSze, buckDeLSVonsag}.

\paragraph{Acknowledgements}  The first author is supported by a Sloan fellowship, and acknowledges the support of the NSF under NSF-DMS 1700312.  The second author acknowledges NSF support under the UT Austin RTG grant in analysis, NSF-DMS 1840314. 

\section{SQG-Reynolds flows}
To state the Main Lemma we will introduce a notion of an SQG-Reynolds flow.  The key difference compared to \cite{buckShkVicSQG} is that we work at the level of the scalar field $\th$, and the error tensor is required to have a double-divergence form.  We will consistently employ the summation notation for repeated indices.
\begin{defn}\label{def: SQG}
	A scalar-valued function $\th: \R \times \T^2 \rightarrow \R$ and a symmetric, traceless tensor field $R^{jl}: \R \times \T^2 \rightarrow \R^{2 \times 2}$ satisfy the \textbf{SQG-Reynolds Equations} if 
	\ALI{
	\pr_t \th + u^l \nb_l \th &= \nb_j \nb_l R^{jl}\\
\mbox{div } u =	\nb_l u^l &= 0\\
	u^l & = T^l \th,
	}
where the operator $T$ is given as in \eqref{eq:SQG} by the Fourier multiplier $m^l(\xi) \coloneqq i \ep^{la} \xi_a / |\xi|$ for $\xi \in \hat{\R}^2$. Here $\ep^{la}$ is the Levi-Civita symbol in two dimensions. By convention this symbol is defined as follows
\begin{equation}\label{def:levi-cevita}
\ep^{11} = \varepsilon^{22} = 0 \quad , \quad
\ep^{12} = 1 \quad , \quad
\ep^{21} = -1
\end{equation} 
Any solution to the SQG-Reynold's equation, $(\th, u, R)$, is called an \textbf{SQG-Reynolds Flow}. The symmetric and traceless tensor field $R^{jl}$ is called the \textbf{stress tensor}. 

At times we will write $(\th, R)$ to refer to the SQG-Reynolds flow, implying that $u^l = T^l[\th]$.
\end{defn}
\subsection{Frequency and Energy Levels}
We will also use a notion of frequency-energy levels similar to those used with the Euler equations in \cite{isett} but with the difference that we assume $\th$ to be frequency localized similar to \cite{buckShkVicSQG}.  A key point is that the relevant fields are measured relative to the size of the stress tensor $R^{jl}$.
\begin{defn}\label{def: freq and energy levels}
	Let $(\th, u, R)$ be a solution of the SQG-Reynolds equation, $\Xi \geq 1$, and $\D_u \geq \D_R \geq 0$ be non-negative numbers. Define the \textbf{advective derivative} with respet to $\th$ to be $D_t \coloneqq \pr_t + T^l \th \nb_l$. We say that $(\th, u, R)$ has \textbf{frequency-energy levels} below  $(\Xi, \D_u, \D_R)$ to order $L$ in $C^0$ if $\th$ and $R$ are of class $C_t^0 C_x^L(\R \times \T^2)$ and the following statements hold
\ALI{
	\supp \hat{\th} &\subseteq \{ \xi: |\xi| \leq \Xi\}\\
	\co{\nb_{\va}\th} , \co{\nb_{\va}u} &\leq \Xi^{|\va|}e_u^{1/2} & \text{for all }|\va | = 0 , \ldots, L\\
	\co{\nb_{\va}R} &\leq \Xi^{|\va|}\D_R & \text{for all }|\va | = 0  , \ldots, L \\
	\co{\nb_{\va}D_t\th}, \co{\nb_{\va}D_t u} &\leq \Xi^{|\va|} (\Xi e_u^{1/2}) e_u^{1/2} & \text{for all }|\va | = 0 ,\ldots, L - 1 \\
	\co{\nb_{\va}D_t R} &\leq \Xi^{|\va|} (\Xi e_u^{1/2}) \D_R & \text{for all }|\va | = 0 , \ldots, L - 1 
}
The quantity $(\Xi e_u^{1/2})^{-1}$ is the \textbf{natural time scale}, the term $e_u^{1/2} \coloneqq \Xi^{1/2}\D_u^{1/2}$, and $\nb$ refers only to derivatives in the spatial variables. We also define a second quantity $e_R \coloneqq \Xi \D_R$ that, like $e_u$, we think of as having units of ``energy density''. 
\end{defn}
\section{Main Lemma}
\begin{lem}[Main Lemma] \label{lem: main}
For $L \geq 2$ there exists a constant $\hc$ such that the following holds: Given an SQG-Reynolds flow $(\th, u, R)$ with frequency and energy levels below $(\Xi, \D_u , \D_R)$ and a non-empty closed interval, $J$, with $\suppt R \subseteq J \subseteq \R$. Let
\[
	N \geq 
	\left (
	\frac{\D_u}{\D_R}	
	\right )
\]
Then there exists an SQG-Reynolds flow $(\ost \th, \ost u, \ost R)$ of the form $\ost \th = \th + \Theta, \ost u = u + T[\Theta]$ with frequency and energy levels bounded by
\[
	(\ost \Xi, \ost \D_u, \ost \D_R) = (\hc N \Xi, \D_R, G \D_R ) , \quad 
	\text{where }G \coloneqq 
	\left(
	\frac{\D_u^{1/4}}{\D_R^{1/4} N^{3/4}}
	\right)
\]
to order $L$ in $C^0$. Furthermore the new stress $\ost R$ and the correction $\Theta$ are supported in time in a neighborhood of $J$
\begin{equation}
	\label{eq: time support growth}
	\suppt \ost R \cup \suppt \Theta \subseteq N(J) \coloneqq
	\{
		t+h : t \in J, |h| \leq 3\nat^{-1}
	\} 
\end{equation}
Additionally one may arrange that $\La^{-1/2}\Theta$ has the form $\La^{-1/2}\Theta = \nb_i W^i$ that satisfies the following statements
\ALI{
	\co{\nb_{\va} \Lambda^{-1/2}\Theta} & \leq \hc (N\Xi)^{|\va|}\D_R^{1/2} , \quad \text{for }|\va| \leq 1 \\
	\co{W} & \leq \hc (N\Xi)^{-1} \D_R^{1/2} \\
	\co{\pr_t \La^{-1/2}\Theta} & \leq \hc \D_R^{1/2} (N\Xi e_u^{1/2})
}
where $\eu$ is defined in Definition \ref{def: freq and energy levels}.  
\end{lem}
\section{Proof of Main Lemma}
\subsection{Shape of the Scalar Corrections}\label{sec: corrections}
Our correction $\Theta$ is a sum of scalar valued waves $\Theta_I$ that oscillate at a large frequency $\la$, similar to the corrections defined in \cite[Section 4.2]{isett2017nonuniqueness} or \cite[Section 5.1]{isettVicol}. 
\begin{equation}
	\label{def: scalar correction}
	\Theta \coloneqq \sum_I \Theta_I(x, t) , \quad
	\Theta_I(x, t) \coloneqq P_\la^I \left (e^{i\la \xi_I}\th_I \right )
	=
	e^{i\la \xi_I}\left( \th_I  + \delta 
	\th_I \right ).
\end{equation}
The \textbf{frequency parameter} $\la \in 2 \pi \N$ is on the order of $B_\la N\Xi$.  More precisely,  
\[
	\la \in [ B_\la N\Xi, 2B_\la N \Xi] \cap (2 \pi \N)
\]
where $B_\la \geq 1$ a very large constant associated to $\la$ that is chosen in Section \ref{sec: verifying the main lemma}. The term $\xi_I \coloneqq \xi_I(x, t)$ is a non-linear \textbf{phase function}, $\th_I \coloneqq \th_I(x, t)$ is the \textbf{amplitude}, and $\delta \th_I \coloneqq \delta \th_I(x, t)$ is a small correction to ensure that $\Theta_I$ is compactly supported in frequency space.  
The second equality for $\Theta_I$ and the explicit formula for $\delta \th_I$ comes from an application of the Microlocal Lemma \ref{lem: microlocal lemma} while the other components of the correction are specified below. 
\subsubsection{The Index Set} \label{sec:Index set}
The index $I$ for the wave $\Theta_I$ is a tuple $I = (k, f) \in \Z \times F$ where $F \coloneqq \{ \pm(1, 2), \pm (2, 1)\}$. Furthermore we define the set $\F \coloneqq F / (+, - ) =   \{ (1, 2), (2, 1)\}$. The index $k$ specifies the interval of time support for the wave $\Theta_I = \Theta_{(k, f)}$ and $f$ represents an initial direction of oscillation.  For each index $I = (k, f) \in \Z \times F$ there is a conjugate index $\bar{I} = (k , \bar{f}) \in \Z \times F$ such that in the summation \eqref{def: scalar correction} each wave $\Theta_I$ has a conjugate wave $\Theta_{\overline{I}} \coloneqq \overline{\Theta}_I$ with an opposite sign phase function $\xi_{\overline{I}} \coloneqq -\xi_I$ and matching amplitude $\th_{\overline{I}} \coloneqq \overline{\th}_I$ so that the overall summation \eqref{def: scalar correction} is real-valued. We also define the notation 
\[
	[k] \coloneqq
	\begin{cases}
		0 & \text{if } k \mod 2 \equiv 0\\
		1 & \text{if } k \mod 2 \equiv 1\\
	\end{cases}	
\]
\subsubsection{Phase functions}\label{sec:phase functions}
For a given SQG-Reynold's flow $\theta$ and each index $I = (k, f)$, the phase functions $\xi_I$ are solutions to the transport equation
\[
	(\pr_t + T^l[\theta]\nb_l ) \xi_I = 0 \quad , \quad
	\xi_I(t(I), x) = \hat{\xi}_I(x)
\]
where the initial data is 
$
	\hat{\xi}_I(x) = \hat{\xi}_{(k, f)}(x) \coloneqq J^k f\cdot x
$
for $J \coloneqq$ a $90^\circ$ rotation. The initial time is $t(I) \coloneqq k\tau$ where $\tau$ is the time step from Section \ref{sec:Time cutoffs}. Furthermore, for a small fixed constant $c_2 > 0$  the time step $\tau$ in Section \ref{sec:Time cutoffs} will be sufficiently small to ensure the phase function gradients 
stay close to their initial data, and stay bounded away from 0 
\begin{equation}
	\label{eq: phase function conditions}
	\left | \nb \xi_I - \nb \hat{\xi}_I \right | \leq c_2, \quad
	\left | \nb \xi_I \right | \geq 1 > 0.
\end{equation}
The constant $c_2$ is determined later by calculation in \eqref{def: c_2}.  It is taken to be sufficiently small to ensure that the phase gradients remain close to their initial conditions and stay in a neighborhood where the angular frequency localizing operator described in Section \ref{sec:freq localizing op} is 1.  In particular $c_2 < 1/4$. 
\subsubsection{Time cutoffs}\label{sec:Time cutoffs}
Define a family of time cutoffs similar to what is done in \cite[Section 5.2]{isettVicol}. For a given index $I = (k, f)$ define the \textbf{time step} $\tau \coloneqq \nat^{-1}b$ where $b$ is a small fraction defined by $b \coloneqq \left ( \frac{\D_u^{1/2}}{\D_R^{1/2}B_\la^{3/2} N^{3/2}}\right )^{1/2}b_0$ and $b_0 \leq 1$ is a uniform constant choosen to make $\tau$ sufficiently small enough to satisfy the conditions of \eqref{eq: phase function conditions}. Note that $\tau$ is smaller than the natural time scale given in Section \ref{def: freq and energy levels}. Additionally, we define time cutoff functions as elements of a rescaled partition of unity in time
\[
	\phi_k(t) \coloneqq \phi \left( \frac{t - k\tau}{\tau}\right), \quad \text{where } \sum_{u \in \Z}\phi^2(t - u) = 1
\]
By our choice of $\phi_k$ each scalar correction $\Theta_I$ will be supported in a time interval $[k\tau - \frac{2}{3}\tau, k\tau + \frac{2}{3}\tau]$. The size of the time interval is choosen to optimize the size of the resulting errors between the transprort stress and high frequency interference stress.
\subsubsection{Angular Frequency Localizing Operator}\label{sec:freq localizing op}
The operator $P_\la^I$ localizes to frequencies of order $\la$ in a neighborhood around $\la \nb \hat{\xi}_I$. This is identical to the frequency localizing operator used in \cite[Section 5.1]{isettVicol}. More precisely, for a function $f \in C^0(\T^2)$ and an index $I$ we define frequency localizing operators in the following way:
\[
	P^I_{ \la} f \coloneqq \int_{\R^2}f(x - h)\chi_{ \la}^I(h) dh
\]
where $\hat{\chi}^I_\la(\xi) \coloneqq \hat{\chi}^I_1(\la^{-1} \xi)$ and we take $\hat{\chi}^I_1 \in C_c^\infty( B_{|\nb \hat{\xi}_I |/2}( \nb \hat{\xi}_I ))$ to be a smooth bump function such that $ \hat{\chi}^I_1 (\xi) =1$ if $|\xi - \nb \hat{\xi}_I| \leq |\nb \hat{\xi}_I | /4$. Moreover it follows from this definition that the corrections $\Theta_I$ and $\Theta$ have frequency support contained in the annulus of order $\la$. That is
\[
	\supp \hat{\Theta}_I \cup \supp \hat{\Theta} \subseteq \{ \xi \in \hat{\R}^2 : \la/2 \leq |\xi | \leq 2\la \}
\] 
\subsubsection{Lifting Function}\label{sec: lifting function}
A lifting function is used in the calculation of the low-frequency stress error in Section \ref{sec: algebra of the cancellation} below. Its purpose is to scale the wave corrections so that they are larger than the low-frequency stress error that is approximately of size $\D_R$. Let $\tau$ be the time step from Definition \ref{def: freq and energy levels} and let $K$ be a large constant to be determined later in the calculation \eqref{def: K}.  We choose $e(t): \R \rightarrow \R_{\geq 0}$ to be a function such that $e(t) \geq K \D_R$ for all $t \in \{ s + h | s \in J, |h| \leq  \ep_t + \nat^{-1}\}$, where $\ep_t$ is the flow mollification parameter defined below in \eqref{def: flow mollification parameters}, such that the bounds 
\[
	\left \| \left( \frac{d}{dt}\right)^r e(t)^{1/2}\right \|_{C^0} \leqc \nat^r \D_R^{1/2}, \quad 0 \leq r \leq 2
\]
hold, and such that $e$ has time support $\supp_t e(t) \subseteq N(J) \coloneqq \{ t+h : t \in J, |h| \leq 3\nat^{-1}\}$.  A simple way to construct $e(t)$ is to define $N_0(J) = \{ t+h : t \in J, |h| \leq 2\nat^{-1}\}$, take the characteristic function of $N_0(J)$, mollify in time with a non-negative convolution kernel $\eta_{\nat^{-1}}$ supported in $\{ |t| \leq (1/3)\nat^{-1} \}$, and multiply by $(K \D_R)^{1/2}$ to construct $e^{1/2}(t)$:
\ALI{
e^{1/2}(t) &= (K \D_R)^{1/2} \eta_{\nat^{-1}} \ast {\bf 1}_{N_0(J)}(t).
}
We remark that the last time support condition holds because $\ep_t \leq \nat^{-1}$, which is verified in Section \ref{sec: verifying the main lemma} below. 
\subsubsection{Amplitudes}\label{sec: amplitudes}
Given an index $I = (k, f)$, the amplitude $\th_I(x,t)$ has the form 
\begin{equation}\label{def: theta_I}
	\th_I(x, t) =  \la^{1/2}\gamma_I(x, t) \phi_k(t) e^{1/2}(t).
\end{equation}
The factors $e(t)$ and $\phi_k(t)$ are the previously defined lifting function and time-cutoff. The functions $\gamma_I: \T^2 \times \R \rightarrow \R$ are called the {\bf coefficients}, and are specified in Section~\ref{sec: algebra of the cancellation}.
In particular $\gamma_I$ depends on the choice of initial directions $\F$ from \ref{sec:Index set} and we construct it to be close in absolute value to $1$ for the duration of the time interval $\phi_k(t)$ is supported on.
\subsection{Shape of the Velocity Corrections}\label{sec: velocity corrections}
By applying the Microlocal Lemma \ref{lem: microlocal lemma} with the convolution operator $T^l P_\la^I$, we can calculate the drift velocity correction $U_I \coloneqq U_I(x, t): \T^2 \times \R \rightarrow \R^2$ from $\Theta_I$:  
\[
	U_I^l \coloneqq T^l\Theta_I 
	=
	T^l P_\la^I \left[ e^{i\la \xi_I}\th_I \right] \\
	=  e^{i\la \xi_I} \left( \th_I m^l(\nb \xi_I) + \delta u_I^l \right).\\
\]
We define
\[
	U_I^l \coloneqq e^{i\la \xi_I} \left( u^l_I + \delta u_I^l \right) , \quad
	u_I^l \coloneqq \th_I m^l(\nb \xi_I)
\]
with an explicit error term $\delta u_I^l$ given by the Microlocal Lemma \ref{lem: microlocal lemma}.  Here we have used that $m^l$ is homogeneous of degree zero to obtain $m^l(\la \nb \xi_I) = m^l(\nb \xi_I)$, and have used that the frequency cutoff in the symbol of $P_\la^I$ is equal to $1$ at the point $\la \nb \xi_I$, due to the condition $|\nb \xi_I - \nb \hat{\xi}_I| \leq c_2$ in \eqref{eq: phase function conditions}.

\subsection{The Microlocal Lemma}
We borrow the following lemma from \cite[Section 4]{isettVicol}, which shows that a convolution operator applied to a highly oscillatory function is multiplication operator to leading order. In all of our applications the kernel $K^l(h)$ will be a Schwartz function essentially supported online scales of order $|h| \sim \la^{-1}$.  The Fourier-transform of a function $K^l:\R^2 \rightarrow \C$ is normalized to be $\hat{K}^l(\xi) = \int_{\R^2}e^{-i\xi \cdot h}K^l(h)dh$. 
\begin{lem}[Microlocal Lemma]\label{lem: microlocal lemma}
	Suppose that $T^l[\Theta](x) = \int_{\R^2}\Theta(x - h)K^l(h)dh$ is a convolution operator acting on functions $\Theta : \T^2 \rightarrow \C$ with a kernel $K^l: \R^2 \rightarrow \C$ in the Schwartz class. Let $\xi : \T^2 \rightarrow \C$ and $\th : \T^2 \rightarrow \C$ be smooth functions and $\la \in \Z$ be an integer. Then for any input of the form $\Theta = e^{i\la \xi(x)}\theta(x)$ we have the formula 
\[
	T^l[\Theta](x) = 
	e^{i\la \xi(x)}
	\left ( 
	\theta(x) \hat{K}^l(\la \nb \xi(x) ) + \delta [T \Theta])(x)
	\right )
\]	
where the error in the amplitude term has the explicit form
\ALI{
	\delta[T^l\Theta](x) &= 
	\int_0^1 dr \frac{d}{dr} \int_{\R^2}e^{-i\la \nb \xi(x)\cdot h}e^{iZ(r,x,h) )}\theta(x - rh)K^l(h)dh\\
	Z(r,x,h) &= r\la \int_0^1 h^ah^b\pr_a\pr_b \xi(x - sh)(1 - s)ds
}
\end{lem} 
The proof of this lemma is given in \cite[Section 4]{isettVicol}.
\subsection{Mollification of the Stress Tensor}
\label{sec: Mollification}
For the stress tensor, $R$, we define its spatial mollification $R_{\ep_x}$ by $R^{jl}_{\ep_x} \coloneqq \chi_{\epsilon_x} \ast \chi_{\epsilon_x} \ast R^{jl}$, with a Schwartz kernel at length scale $\ep_x$ so that $\chi_{\epsilon_x}(h) = \ep_x^{-2}\chi_1(\ep_x^{-1}h)$ where $\chi_1: \R^2 \rightarrow \R$ is a Schwartz kernel with integral 1 with a vanishing moment condition that is $\int_{\R^2}h^{\va} \chi_{\epsilon_x}(h)dh = 0$ for all multi-indices $1 \leq |\va | \leq L$. We next define the mollified stress tensor $R_\epsilon$ by a mollification of $R_{\epsilon_x}$ along the coarse scale flow in time:
\[
	R_\epsilon^{jl} = 
	\eta_{\epsilon_t} \ast_\Phi R^{jl}
	\coloneqq
	\int_\R R^{jl}_{\epsilon_x}(\Phi_s(t, x)) \eta_{\epsilon_t}(s)ds
\]
Where $\eta_{\epsilon_t}(s) = \epsilon_t^{-1}\eta_1(\epsilon_t^{-1}s)$ is a smooth, positive mollifying kernel with compact support in the interval $|s| \leq \epsilon_t \leq \nat^{-1}$ and $\int_\R \eta_{\epsilon_t}(s) ds = 1$, while $\Phi_s(t, x)$ is the flow map of $\pr_t  + u \cdot \nb$, which takes values in $\R \times \T^2$ and is defined as the unique solution to the ODE
\[
	\Phi_s(t, x) = (t + s, \Phi^i_s(t, x)), \quad 
	\Phi_0(t, x) = (t, x), \quad
	\frac{d}{ds}\Phi^i_s(t, x) = u^i(\Phi_s(t, x)).
\] 
We choose the mollification parameters to be
\begin{equation}\label{def: flow mollification parameters}
	\epsilon_x = c_0N^{-\fr{3}{2L}}\Xi^{-1}, \quad
	\epsilon_t = c_0 N^{-\fr{3}{2}}\Xi^{-3/2}\D_R^{-1/2}
\end{equation}
By this choice of parameters we may obtain bounds that are similar to the bounds given in \cite{isett}. Finally, the constant $c_0$ will only depend on the Schwartz kernels $\eta_1, \chi_1$. We state $C^0$ norm estimates for $R_\ep$ in Section \ref{sec: basic estimates} below.  
\subsection{The Error Terms} \label{sec: Error terms}
Let $(\theta, u, R)$ be an SQG-Reynolds flow as defined in \ref{def: SQG}. We will construct the new  scalar field $\ost \theta \coloneqq \theta + \Theta$ and $\ost u = T[\ost \theta]$ by adding a high frequency correction $\Theta$. The new Reynolds stress solves 
\[
	\nb_j \nb_l \ost R^{jl} = \nb_j \nb_l R^{jl}_\epsilon + \nb_j \nb_l \left( R^{jl} - R^{jl}_\epsilon \right )+ \left ( \pr_t \Theta + T^l \theta \nb_l \Theta + T^l \Theta \nb_l \theta \right ) + T^l \Theta \nb_l \Theta
\]
If we recall that the corrections are of the form $\Theta = \sum_I \Theta_I$ then we can decompose the new error into the following terms
\ALI{
	\nb_j \nb_l \ost{R}^{jl} &= \nb_j \nb_l \left ( R_T^{jl} + R_H^{jl} + R_M^{jl} + R_S^{jl} \right ) \\
	\text{where }
	\nb_j \nb_l R_H^{jl} &= \sum_{J,I \in \Z \times F, J \neq \overline{I}} T^l [\Theta_I] \nb_l\Theta_{J} + T^l [\Theta_J] \nb_l\Theta_{I} \\
	\nb_j \nb_l R_T^{jl} &= \pr_t \Theta + T^l \theta \nb_l \Theta + T^l \Theta \nb_l \theta \\
	\nb_j \nb_l R_M^{jl} &= \nb_j \nb_l \left( R^{jl} - R^{jl}_\epsilon \right )\\
	\nb_j \nb_l R_S^{jl} &=  \nb_j \nb_l R^{jl}_\epsilon + \nb_l \sum_{I \in \Z \times F} T^l [\Theta_I] \Theta_{\overline{I}} 	
}
We refer to the term $R_S$ as the \textbf{Low frequency stress error}, $R_T$ as the \textbf{Transport stress error}, $R_H$ as the \textbf{High frequency interference stress error}, and $R_M$ as the \textbf{Mollification stress error}. These terms are symmetric $(2, 0)$-tensors. Let $\RR$ be the second order anti-divergence operator defined in \ref{sec:Second order anti-div}, let $B_\la$ be the bilinear anti-divergence operator defined in \eqref{def: Bilinear form}, and let $M_{[k]}$ be the constant matrix defined in \eqref{def: M_k}. We can define the stress errors using these tools as 
\begin{align}
\label{def: New R}
\ost{R}^{jl} &\coloneqq R_H^{jl} + R_T^{jl} + R_M^{jl} + R_S^{jl}  \\
\label{def: R_H}
R_H^{jl} &\coloneqq \sum_{J,I \in \Z \times F, J \neq \overline{I}} \RR^{jl}\left [T^l [\Theta_I] \nb_l\Theta_{J} + T^l [\Theta_J] \nb_l\Theta_{I}\right ] \\
\label{def: R_T}
R_T^{jl} &\coloneqq \RR^{jl} \left [ \pr_t \Theta + T^l \theta \nb_l \Theta + T^l \Theta \nb_l \theta \right  ] \\
\label{def: R_M}
R_M^{jl} &\coloneqq R^{jl} - R^{jl}_\epsilon \\
\label{def: R_S}
R_S^{jl} &\coloneqq 
\sum_{I = (k, f) \in \Z \times \F} B_\la^{jl}[\Theta_I,  \Theta_{\overline{I}}] 
-
\phi_k^2(t) 
\left ( 
e(t)M_{[k]} -
R_\epsilon^{jl}
\right)
\end{align}
The operator $\RR$ above satisfies $\nb_j\nb_l\RR^{jl}[U] = U$ whenever $U$ has mean zero.  Note that the argument of $\RR$ in lines \eqref{def: R_H} and \eqref{def: R_T} has mean zero due to the frequency support of $\pr_t \Th$ and the fact that the other terms can be written in divergence form.  In line \eqref{def: R_S}, we use that $\nb_j \nb_l B_\la^{jl}[\Th_I, \Th_{\bar{I}}] = \nb_l\left[T^l[\Th_I]\Th_{\bar{I}} + T^l[\Th_{\bar{I}}] \Th_I\right]$ and that the spatial divergence of a function of $t$ alone is $0$.

In the following sections we define the operators used above as well as other tools used in estimating the stress tensors. To prove the Main Lemma we will prove estimates on each of the stress errors. What we will show is that for the weighted norm $\ost H[\cdot ]$ defined in Lemma \ref{lem: weighted norm}. 
\[
\ost H[ R_T] \leqc (B_\la N \Xi)^{-3/2} \D_R^{1/2}\nat b^{-1} , \quad 
\ost H[ R_S] \leqc B_\la^{-1}N^{-1}\D_R , \quad 
\ost H[ R_H] \leqc b\D_R , \quad
\ost H[ R_M] \leq \fr{G \D_R}{1000} 
\]
Note that the choice of $b$ in Section \ref{sec:Time cutoffs} optimizes between $R_T$ and $R_H$, which are the largest errors.
\subsubsection{Second Order Anti-Divergence Equation}\label{sec:Second order anti-div}
We use a second order anti-divergence operator to produce a symmetric and traceless (2,0) tensor for the transport stress error and the high-high interfernce error. For all smooth, mean-zero $f: \C^2 \rightarrow \C$, we define the \textbf{second order anti-divergence operator} as the operator 
\begin{equation}\label{def: second order anti-div operator}
	\RR \coloneqq 2\RR_1 -\RR_2
\end{equation}
where $\RR_1$ and $\RR_2$ are the second order anti-divergece operators $\RR^{jl}_1 \coloneqq \nb^j \nb^l \De^{-2}$, $\RR^{jl}_2 \coloneqq \delta^{jl}\De^{-1}$. We note that $\RR^{jl}[f]$ is a symmetric, traceless tensor that satisfies $\nb_j \nb_l \RR^{jl}[f] = f$.
\subsubsection{A Bilinear Anti-Divergence Form}
We use a bilinear first order anti-divergence operator to produce a symmetric, traceless (2,0) tensor for the low frequency stress error.  The work of this section is motivated by \cite[Section 5.4.2]{buckShkVicSQG}.  

Consider the problem of solving $\nb_j R^{j\ell} = Q^l$ for the term 
\ALI{
Q^l(x) \coloneqq T^l[\Theta_I] \Theta_J + T^l[\Theta_J]\Theta_I.
}
Although it is not obvious that an anti-divergence for $Q^l$ would have a good form, there is at least hope of finding an anti-divergence for $Q^l$ as $Q^l$ satisfies the necessary and sufficient condition of having integral $0$ due to the anti-self-adjointness of $T^l$, which is equivalent to the multiplier $m^l$ being odd.  Since the anti-self-adjointness is most easily seen in frequency space, there is motivation to look at the problem in frequency space\footnote{See also \cite{IOnonpd} for a different anti-divergence operator derived using Taylor expansion in frequency space.} to define $R^{jl}$.  

Recall now that the  corrections from Section~\ref{sec: corrections} are constructed with compact frequency support in the annulus $\{ \xi : 2^{-1}\la \leq |\xi| < 2\la \}$.  We start by replacing $\Theta_I, \Theta_J$ with Schwartz approximations that we will also call $\Theta_I, \Theta_J$ by an abuse of notation. By using a mollification in frequency space we may assume that the Schwartz approximations maintain frequency support in $\{ \xi : 2^{-1}\la \leq |\xi| < 2\la \}$, and that they converge pointwise in physical space while also remaining uniformly bounded there. For these approximations we construct a bilinear anti-divergence as follows. Define a smooth bump function $\hat{\chi}_{\approx \la}$ that is 1 on the annulus $\{ \xi : 2^{-1}\la \leq |\xi| < 2\la \}$ so that $\hat{\Theta}_I = \hat{\chi}_{\approx \la} \hat{\Theta}_I$ and $\hat{\Theta}_J = \hat{\chi}_{\approx \la} \hat{\Theta}_J$. Recall from Definition \ref{def: SQG} that $m$ is the multiplier for $T$ and take a Fourier transform of $Q^l(x)$. 
\begin{align}
	\hat{Q}^l(\xi) 
	&= 
	\int_{\hat{\R}^2} 
	m^l(\xi - \eta) \hat{\Theta}_I(\xi - \eta) \hat{\Theta}_J(\eta) + 
	\hat{\Theta}_I(\xi - \eta) m^l(\eta)\hat{\Theta}_J(\eta) \fr{d\eta}{(2\pi)^2} \nonumber \\
	&= 
	\int_{\hat{\R}^2} 
	m^l(\xi - \eta)\hat{\chi}_{\approx \la}\hat{\Theta}_I(\xi - \eta)\hat{\Theta}_J(\eta) + 
	\hat{\Theta}_I(\xi - \eta) m^l(\eta)\hat{\chi}_{\approx \la}\hat{\Theta}_J(\eta) \fr{d\eta}{(2\pi)^2} \nonumber \\
	&\text{Define } m_{\approx \la}(\cdot) \coloneqq m(\cdot)\hat{\chi}_{\approx \la}(\cdot) \nonumber \\
	&=
	\int_{\hat{\R}^2} 
	\left [m^l_{\approx \la}(\xi - \eta) + m^l_{\approx \la}(\eta) \right ] 
	\hat{\Theta}_I(\xi - \eta) 
	\hat{\Theta}_J(\eta) \fr{d\eta}{(2\pi)^2} \nonumber \\
	&\text{By the oddness of the multiplier } m^l_{\approx \la}\nonumber \\
	&=
	\int_{\hat{\R}^2} 
	\left [m^l_{\approx \la}(\xi - \eta) - m^l_{\approx \la}(-\eta) \right ] 
	\hat{\Theta}_I(\xi - \eta) 
	\hat{\Theta}_J(\eta) \fr{d\eta}{(2\pi)^2} \nonumber \\
	&=
	i\xi_j
	\int_{\hat{\R}^2} \int_0^1 
	-i \nb^j m^l_{\approx \la}(u_\si) d\si  
	\hat{\Theta}_I(\zeta) 
	\hat{\Theta}_J(\eta) \fr{d\eta}{(2\pi)^2} \nonumber \\
	&\eqqcolon
	i\xi_j
	\int_{\hat{\R}^2}
	\hat{K}^{jl}_\la (\zeta, \eta) 
	\hat{\Theta}_I(\zeta) 
	\hat{\Theta}_J(\eta) \fr{d\eta}{(2\pi)^2}, \label{eq: hat Q formula}
\end{align}
where we have set
\[
	\zeta \coloneqq \xi - \eta, \quad u_\si \coloneqq \si (\xi - \eta) - (1 - \si)\eta, 
\]
and we have used Taylor's Remainder Theorem in the last equality. In the last line we defined 
\begin{equation}\label{def: hat K }
	\hat{K}^{jl}_\la (\zeta, \eta) \coloneqq 
	\int_0^1 
	-i \nb^j m^l_{\approx \la}(u_\si)
	\hat{\chi}_{\approx \la}(\zeta)
	\hat{\chi}_{\approx \la}(\eta)
	d\si.
\end{equation}
Line~\eqref{eq: hat Q formula} defines in frequency space a solution $R^{jl}$ to $\nb_j R^{jl} = Q^l$ that is not symmetric, whereas our ultimate goal is in fact to define a solution to $\nb_j \nb_l R^{jl} = \nb_l Q^l$ that is symmetric.  Therefore we proceed by taking the symmetric part of the above solution (noting that $\nb_j \nb_l R^{jl}_{\mathrm{asym}} = 0$ for the anti-symmetric part), and taking an inverse Fourier transform.  

Define the symmetric part of this kernel as
\begin{equation}\label{def:hat K sym}
	\hat{K}^{jl}_{\la, \text{sym}} \coloneqq \frac{\hat{K}^{jl}_\la+ \hat{K}_\la^{lj}}{2}.
\end{equation}
It can be shown that $\hat{K}_{\la, \text{sym}}$ is traceless for the specific multiplier of the SQG equation. 
Now consider 
\[
	\hat{Q}^l_2(\xi) \coloneqq 
	i\xi_j
	\int_{\hat{\R}^2}
	\hat{K}^{jl}_{\la, \text{sym}} (\zeta, \eta) 
	\hat{\Theta}_I(\zeta) 
	\hat{\Theta}_J(\eta) \fr{d\eta}{(2\pi)^2}
\]
and take the inverse Fourier Transform to obtain 
\ALI{
	Q_2^l(x) &= \nb_j \int_{\hat{\R}^2}e^{i \xi \cdot x} \int_{\hat{\R}^2}
	\hat{K}^{jl}_{\la, \text{sym}} (\zeta, \eta)
	\hat{\Theta}_I(\zeta) 
	\hat{\Theta}_J(\eta)
	\fr{d\eta}{(2\pi)^2} \fr{d\xi}{(2\pi)^2} \\
	&=
	 \nb_j \int_{\hat{\R}^2\times \hat{\R}^2} e^{i (\zeta + \eta) \cdot x}
	 \hat{K}^{jl}_{\la, \text{sym}} (\zeta, \eta)
	\hat{\Theta}_I(\zeta) 
	\hat{\Theta}_J(\eta)
	\fr{d\eta}{(2\pi)^2} \fr{d\zeta}{(2\pi)^2}.
}
Observe that $\hat{K}_{\la, \text{sym}}^{j\ell}$ is 
compactly supported in frequency and smooth due to the factors of $\hat{\chi}_{\approx \la}$, making $\hat{K}_{\la, \text{sym}}^{j\ell}$ a Schwartz function. Thus we can define the corresponding physical space kernel  $K^{jl}_{\la, \text{sym}} (h_1, h_2)$ as the inverse Fourier Transform of $\hat{K}^{jl}_{\la, \text{sym}} (\zeta, \eta)$ so that 
\[
	\hat{K}^{jl}_{\la, \text{sym}} (\zeta, \eta) = \int_{\hat{\R}^2\times \hat{\R}^2} e^{-i(\zeta, \eta) \cdot (h_1, h_2)}K^{jl}_{\la, \text{sym}} (h_1, h_2) dh_1 dh_2.
\]
Then 
\ALI{
	Q_2^l(x) &=
	\nb_j \int_{\R^2 \times \R^2} 
	e^{i (\zeta + \eta) \cdot x}	
	\int_{\hat{\R}^2\times \hat{\R}^2}
	e^{-i(\zeta, \eta) \cdot (h_1, h_2)}
	K^{jl}_{\la, \text{sym}}(h_1, h_2)
	dh_1 dh_2 
	\hat{\Theta}_I(\zeta) 
	\hat{\Theta}_J(\eta)
	\fr{d\eta}{(2\pi)^2} \fr{d\zeta}{(2\pi)^2} \\
	&=
	\nb_j \int_{\R^2 \times \R^2} 
	K^{jl}_{\la, \text{sym}}(h_1, h_2)
	\int_{\hat{\R}^2}
	e^{i\zeta \cdot (x - h_1)}
	\hat{\Theta}_I(\zeta) 
	\fr{d\zeta}{(2\pi)^2}
	\int_{\hat{\R}^2}
	e^{i\eta \cdot (x - h_2)}
	\hat{\Theta}_J(\eta)
	\fr{d\eta}{(2\pi)^2}
	dh_1 dh_2 \\
	&=
	\nb_j \int_{\R^2 \times \R^2}K^{jl}_{\la, \text{sym}}(h_1, h_2)\Theta_I(x - h_1)\Theta_J (x - h_2) dh_1 dh_2 \\
	&\eqqcolon \nb_j B^{jl}_{\la}[\Theta_I, \Theta_J](x)
}
Here we define the \textbf{Bilinear anti-divergence operator} or \textbf{Bilinear form} as
\begin{equation} \label{def: Bilinear form}
	B^{jl}_{\la}[F_1, F_2](x) \coloneqq  \int_{\R^2 \times \R^2}K^{jl}_{\la, \text{sym}}(h_1, h_2)F_1(x - h_1 )F_2(x - h_2) dh_1 dh_2,
\end{equation}
for scalar fields $F_1, F_2 \in C^\infty(\T^2)$.   
This bi-convolution operator satisfies 
\ali{
T^l[\Theta_I] \nb_l \Theta_J + T^l[\Theta_J]\nb_l \Theta_I =
	\nb_l \nb_j B^{jl}_{\la}[\Theta_I, \Theta_J](x) \label{eq:bilin:antidiv}
}
for all 
Schwartz approximations to $\Theta_I, \Theta_J$ with Fourier support in $\{ |\xi| \sim \la \}$. Since $\|K_{\la, \text{sym}} \|_{L^1(\R^2)}$ is finite, we may pass to the limit in \eqref{def: Bilinear form} from Schwartz approximations using the dominated convergence theorem to conclude that \eqref{eq:bilin:antidiv} holds (both in $\DD'$ and classically) for the periodic functions $\Theta_I, \Theta_J$.  
We note that by scaling considerations we have the following inequality for $h = (h_1, h_2) \in \R^{2} \times \R^2$
\begin{equation}
	\label{eq: Kernel scaling}
\la^m	\| |h|^m K_{\la, \text{sym}}^{j\ell}(h_1, h_2) \|_{L^1(\R^2 \times \R^2)} \leqc_m \la^{-1},
\end{equation}
and we also observe for use in the calculations of Section~\ref{sec:explicitMultip} that, by \eqref{def: hat K },
\ali{
\widehat{K}_{\la, \text{sym}}(\la \nb \xi_I, -\la \nb \xi_I) &= -i(2\la)^{-1} (\nb^j m^l + \nb^l m^j)(\nb \xi_I)  \label{eq:K:nonexplicit}
}
is homogeneous of degree $-1$ in $\la$ and in $\nb \xi_I$.  In the above we have used that $\nb m$ is degree $-1$ homogeneous, and that $\hat{\chi}_{\approx \la}(\la \nb \xi_I) = 1$ while $\nb\hat{\chi}_{\approx \la}(\la \nb \xi_I) = 0$ due to the condition \eqref{eq: phase function conditions}, which keeps the phase gradient within a small distance of its initial conditions.
\subsection{Low frequency part of the error}\label{sec: algebra of the cancellation}
The goal of this section is to explain how the wave corrections described in Section \ref{sec: corrections} are used to cancel the Reynolds stress error in the Low frequency stress error. The tools described in this section are similar to ones used in \cite{isettVicol}. In this section we will define the coefficient functions $\gamma$. 
\subsubsection{The Bilinear Microlocal Lemma}
The lemma below is inspired from an analogous version used in \cite{isettVicol} and stated in \ref{lem: microlocal lemma}. This tool tells us that the bilienar convolution operator acts like a multicative operator to leading order when it is applied to our highly oscillatory scalar field corrections $\Theta$.
\begin{lem}[Bilinear Microlocal Lemma] 
\label{lem: Bilinear Microlocal}
For conjugate scalar field corrections $\Theta_I , \Theta_{\overline{I}}$ as defined in \ref{sec: corrections} and a bilinear form $B_\la$ as defined in \eqref{def: Bilinear form} we have the following identity
\[
	B^{j\ell}_\la[\Theta_I, \Theta_{\overline{I}}](x, t) = 
		\theta_I^2(x, t)\hat{K}^{j\ell}_{\la, \text{sym}}[\lambda \nb \xi_I,-\lambda \nb \xi_I] + \delta B_I^{j\ell}(x, t)
\]
where $\delta B^{j\ell}$ is a small error term that has the explicit form
\ALI{
	\delta B_I^{j\ell}(x, t) &\coloneqq 
	\int_{\R^2\times\R^2}e^{i\la \nb \xi_I(x) \cdot h_1}e^{-i \la \nb \xi_I(x) \cdot h_2}K^{j\ell}_{\la, \text{sym}}[h_1, h_2]
	\th_I(x) Y(x, h_1)
	dh_1 dh_2
	\\
	&+
	\int_{\R^2\times \R^2}e^{i\la \nb \xi_I(x) \cdot h_1}e^{-i \la \nb \xi_I(x) \cdot h_2}K^{j\ell}_{\la, \text{sym}}[h_1, h_2]
	\th_I(x) Y(x, h_2)  
	dh_1 dh_2\\
	&+
	\int_{\R^2\times \R^2}e^{i\la \nb \xi_I(x) \cdot h_1}e^{-i \la \nb \xi_I(x) \cdot h_2}K^{j\ell}_{\la, \text{sym}}[h_1, h_2]
	Y(x, h_1) \cdot Y(x, h_2)
	dh_1 dh_2\\
	Y(x, h) &\coloneqq 
	\int_0^1
	\frac{d}{dr} e^{iZ(r, x, h)} \th_I (x - rh)	
	dr\\
	Z(r, x, h) &\coloneqq 
	r\la \int_0^1 h^a h^b \pr_a \pr_b \xi(x - sh ) (1 - s) ds
}
\end{lem}
\begin{proof}
	The proof is based on the Taylor expansion as in the Microlocal Lemma \ref{lem: microlocal lemma}.   
  The starting point for the computation is the approximation
	\ALI{
	B^{j\ell}_\la[\Theta_I, \Theta_{\overline{I}}] &= \int_{\R^2 \times \R^2} e^{i \la \xi_I(x - h_1)} e^{-i\la \xi_I(x -h_2)}  \th_I(x-h_1) \th_I(x-h_2) K_{\la,\mathrm{sym}}^{jl}[h_1,h_2] dh_1 dh_2 \\
	=\int_{\R^2 \times \R^2} &\left[e^{i \la (\xi_I(x - h_1) - \xi_I(x))} \th_I(x-h_1)\right] \left[e^{-i\la(\xi_I(x -h_2) - \xi_I(x))} \th_I(x-h_2) \right] K_{\la,\mathrm{sym}}^{jl}[h_1,h_2] dh_1 dh_2 \\
	&\approx \th_I(x)^2 \int_{\R^2 \times \R^2} e^{-i \la \nb \xi_I \cdot h_1} e^{i \la \nb \xi_I \cdot h_2} K_{\la,\mathrm{sym}}^{jl}[h_1,h_2] dh_1 dh_2,
	}
where we recognize that the last integral is exactly $\widehat{K}_{\la, \mathrm{sym}}^{jl}[\la\nb \xi_I(x), -\la\nb \xi_I(x)]$.
	\end{proof}
\subsubsection{Explicitly Calculating the Symmetric Multiplier} \label{sec:explicitMultip}
We explicitly calculate the $\hat{K}_{\la, \text{sym}}^{j\ell}$ term.  
Starting from $m^l(p) = i \ep^{\ell a} p_a/|p|$ and \eqref{eq:K:nonexplicit}, we have 
\begin{align}\label{eq: evaluated hat K}
	\hat{K}_{\la, \text{sym}}^{j\ell}( \la \nb \xi_I, - \la \nb \xi_I) 
	&=
	-2^{-1} \la^{-1}
	i\left ( \nb^jm^\ell +\nb^\ell m^j\right ) (\nb\xi_I)
	\\
	&= 
	(-2\la)^{-1} 
	\left ( \frac{\ep^{\ell a}  \nb_a \xi_I  \nb^j \xi_I + \ep^{ja}  \nb_a \xi_I  \nb^\ell \xi_I }{| \nb \xi_I|^3} \right ) \nonumber 
	\\
	\hat{K}_{\la, \text{sym}}( \la \nb \xi_I, - \la \nb \xi_I)&=
	(2\la)^{-1}
	\begin{bmatrix}
		-2 \nb_1 \xi_I \nb_2 \xi_I & (\nb_1 \xi_I)^2 - (\nb_2 \xi_I)^2  \\
		 (\nb_1 \xi_I)^2 - (\nb_2 \xi_I)^2 & 2 \nb_1 \xi_I \nb_2 \xi_I
	\end{bmatrix}
	| \nb \xi_I |^{-3} \nonumber
\end{align}
In particular, the result is traceless and homogeneous of degree $-1$ in $\la$ and the phase gradients.
\subsubsection{Solving the Quadratic equation using Linear Phase Functions}

Our goal now is to define coefficients $\ga_I$ such that the low frequency error term \eqref{def: R_S} is close to zero.

We will treat the term $R_\ep$ as a lower order term, and will treat the phase gradients $\nb \xi_I$ as perturbations of their initial conditions $\nb \hat{\xi}_I$, which are constants.  Therefore, employing the approximation 
\ALI{
B_\la[\Th_I, \Th_{\bar{I}}] \approx \th_I^2 \widehat{K}_{\la, \mathrm{sym}}^{jl}[\la \nb \xi_I,-\la \nb \xi_I] \approx  \th_I^2 \widehat{K}_{\la, \mathrm{sym}}^{jl}[\la\nb \hat{\xi}_I,-\la\nb \hat{\xi}_I]
}
we are motivated by \eqref{def: R_S} to consider the following linear system 
\begin{equation}\label{eq:linear system}
\sum_{I \in \{k \}\times \F} \theta_I^2(x, t)\hat{K}^{j\ell}_{\la, \text{sym}}[\lambda \nb \hat{\xi}_I,-\lambda \nb \hat{\xi}_I] 
= \phi_k(t) e(t) M_{[k]}
\end{equation}
In this case, the $\gamma_I(x, t) = \hat{\gamma}_I$ will be constant coefficients and we assume the linear initial conditions $\nb \hat{\xi}_I$ defined in Section \ref{sec:phase functions}.  We define a set of constant matrices $M_{[k]}$ as  
\[
	M_{[k]} \coloneqq
	2^{-1}
	5^{-3/2}
	\begin{bmatrix}
	-8 & 0 \\ 0 & 8
	\end{bmatrix}
	\; 
	\text{if }k\equiv 0 \mod 2 ,
	\quad
	M_{[k]} \coloneqq
	2^{-1}
	5^{-3/2}
	\begin{bmatrix}
	8 & 0 \\ 0 & -8
	\end{bmatrix}
	\; 
	\text{if }k \equiv 1 \mod 2. \quad
\]
For any index $I = (k, f) \in \Z \times \F$ as defined in Section \ref{sec:Index set} we can explicitly calculate the left hand tensor of $\eqref{eq:linear system}$ in terms of
\begin{align}
	\label{def: M_k}
	\hat{K}_{\text{sym}}^{j\ell}( \nb\hat\xi_{(k, J^k(1,2))}, - \nb\hat\xi_{(k, J^k(1, 2))}) &= 
	\begin{cases}
		2^{-1}
		5^{-3/2}
		\begin{bmatrix}
			-4 & -3 \\ -3 & 4
		\end{bmatrix}
		& 
		\text{if } [k] = 0\\
		2^{-1}
		5^{-3/2}
		\begin{bmatrix}
			4 & 3 \\ 3 & -4
		\end{bmatrix}
		&
		\text{if } [k] = 1
	\end{cases}\\
	\nonumber
	\hat{K}_{\text{sym}}^{j\ell}( \nb\hat\xi_{(k, J^k(2, 1))}, - \nb\hat\xi_{(k, J^k(2, 1))}) &= 
	\begin{cases}
		2^{-1}
		5^{-3/2}
		\begin{bmatrix}
			-4 & 3 \\ 3 & 4
		\end{bmatrix}
		& 
		\text{if } [k] = 0\\
		2^{-1}
		5^{-3/2}
		\begin{bmatrix}
			4 & -3 \\ -3 & -4
		\end{bmatrix}
		&
		\text{if } [k] = 1.
	\end{cases}
\end{align}
We observe that in either the case $[k] = 0$ or $[k] = 1$ the matrix pair is linearly independent and hence crucially spans the two dimensional space of symmetric, traceless $2\times 2$ matrices. We also note that 
\[
	\hat{K}^{j\ell}_{\la, \text{sym}}[\la \nb \hat{\xi}_I,-\la \nb \hat{\xi}_I] 
	= 
	\la^{-1}\hat{K}^{j\ell}_{\text{sym}}[ \nb \hat{\xi}_I,- \nb \hat{\xi}_I]. 
\]
By applying this identity to \eqref{eq:linear system} and recalling $\theta^2_I = \lambda\gamma^2_Ie(t)\phi_{k}(t)$ from Section \ref{sec: amplitudes} then for any $k \in \Z$ solving \eqref{eq:linear system} becomes equivalent to solving the following equation of matrices
\begin{align}
	\label{eq: M_0}
	M_{[k]} =&
	\hat{\gamma}_{([k], J^k(1, 2))}^2	
	\hat{K}_{\text{sym}}^{j\ell}( \nb\hat\xi_{(k, J^k(1,2))}, - \nb\hat\xi_{(k, J^k(1, 2))}) \\
	+&
	\nonumber
	\hat{\gamma}_{([k], J^k(2,1))}^2
	\hat{K}_{\text{sym}}^{j\ell}( \nb\hat\xi_{(k, J^k(2, 1))}, - \nb\hat\xi_{(k, J^k(2, 1))})
\end{align}
The equation of matrices above is solvable by using constant coefficients $\hat{\gamma}_{([k], J^k(1,2))}^2 = \hat{\gamma}_{([k], J^k(2,1))}^2 = 1$. 
\subsubsection{Solving the Cancellation as a Perturbation of a Linear System}
As typical with convex integration schemes we will perform the error cancellation on a mollified version of the stress tensor $R_\epsilon$ as defined in \ref{sec: Mollification}. Please refer to \cite[Section 18]{isett} for detailed discussion of the mollified stress tensor. Consider the term $\nb_j \nb_j R_S^{jl}$ from Section \ref{sec: Error terms}
\begin{align}
\label{eq: double div of R_S}
	\nb_j\nb_\ell R_S^{j\ell} &= 
	\nb_j \nb_l R^{jl}_\epsilon + \nb_l \sum_{I = (k, f) \in \Z \times F} T^l [\Theta_I] \Theta_{\overline{I}} 
\end{align}
Let $e(t)$ be the lifting function defined in Section \ref{sec: lifting function}.  
Also let $M_{[k]}$ be the constant matrix defined in \eqref{def: M_k} and let $B_\la$ be the Bilinear form defined in \eqref{def: Bilinear form} then $\eqref{eq: double div of R_S}$ is equivalent to 
\[
	\eqref{eq: double div of R_S} = 
	\nb_j\nb_\ell 
	\sum_{I = (k, f) \in \Z \times \F}  B_\la^{jl}[\Theta_I,  \Theta_{\overline{I}}] - \phi_k^2(t) 
	\left (
	e(t)M_{[k]} -
	R_\epsilon^{jl}
	\right)
\]
If we drop the double divergence from the above expression then we recover the definition of $R_S$ from \eqref{def: R_S}. We can then apply the Bilinear Microlocal Lemma \ref{lem: Bilinear Microlocal} to $R_S$ to produce
\begin{equation}
	\label{eq: R_S sum}
	R_S = 
	\sum_{I = (k, f) \in \Z \times \F} \theta_I^2(x, t)\hat{K}_{\la, \text{sym}}^{j\ell}[\lambda \nb \xi_I,-\lambda \nb \xi_I] 
	- \sum_k
	e(t)
	\phi_k^2(t)
	\left (	
	M_{[k]}
	-
	\fr{R_\epsilon^{j\ell}}{e(t)}
	\right )
	+ \sum_I \delta B_I^{jl}(x, t)
\end{equation}
We will show that on every time interval, indexed by $k$, the first two summands add to zero pointwise in $(x, t)$. Without loss of generality consider any single time interval indexed by $k$. We will prove 
\begin{equation}
	\label{eq: perturbed linear system}
	\sum_{I \in \{k \}\times \F} \theta_I^2(x, t)\hat{K}_{\la, \text{sym}}^{j\ell}[\lambda \nb \xi_I,-\lambda \nb \xi_I] 
	-
	e(t) \phi_k^2(t)
	\left (
	M_{[k]}
	-
	\fr{R_\epsilon^{j\ell}}{e(t)}
	\right)
	=
	0
\end{equation}
So recall $\theta^2_I = \lambda\gamma^2_Ie(t)\phi_k^2(t)$ from Section \ref{sec: amplitudes} and that $\hat{K}_{\text{sym}}^{j\ell}$ has homogeneity order $-1$ from \eqref{eq: evaluated hat K}. By factoring out appropriately can see that the left hand side of \eqref{eq: perturbed linear system} vanishes if we can solve 
\begin{equation}
\label{eq: gamma system}
	\sum_{I \in \{k\}\times \F} \gamma^2_I \hat{K}_{\la, \text{sym}}^{j\ell}[\nb \xi_I,-\nb \xi_I] = M_{[k]}^{j\ell} - \varepsilon^{jl} , \quad \text{where }\varepsilon^{jl} \coloneqq \frac{R_\epsilon^{j\ell}}{e(t)}
\end{equation}
We view \eqref{eq: gamma system} as a perturbation of our solved linear system \eqref{eq:linear system}. Let $\mathring{\SS}_{2\times 2}$ be the space symmetric, traceless $(2, 0)$ tensors. For parameters $(p_1, p_2)$ define the linear map $\LL_{(p_1, p_2)}: \R \times \R \longrightarrow \mathring{\SS}_{2\times 2}$
\ALI{
	\LL_{(p_1, p_2)}
	(x_1, x_2)
	=
	\left (
	x_1 \hat{K}_{\la, \text{sym}}^{j\ell}[p_1,-p_1]
	+
	x_2 \hat{K}_{\la, \text{sym}}^{j\ell}[p_2,-p_2] 
	\right ) 
}
If we set
\ALI{
p_1 = \nb \hat{\xi}_{I_1}, \quad p_2 	= \nb \hat{\xi}_{I_2} \qquad 	I_1 = [k] \times J^k(1, 2), \quad I_2 = [k] \times J^k(2, 1)
}
then by the calculation \eqref{eq: M_0} we know that 
$
	\LL_{(\nb \hat{\xi}_{I_1}, \nb \hat{\xi}_{I_2})}
	(\hat{\gamma}^2_{I_1}, \hat{\gamma}^2_{I_2})
	=
	\LL_{(\nb \hat{\xi}_{I_1}, \nb \hat{\xi}_{I_2})}
	(1, 1)
	=
	M_{[k]}
$. 
Moreover for the parameters $\nb \hat{\xi}_{I_1},$ and $\nb \hat{\xi}_{I_2}$, we know from the linear independence of the matrices in \eqref{def: M_k} that $\LL_{(\nb \hat{\xi}_{I_1}, \nb \hat{\xi}_{I_2})}$ maps a basis in $\R^2$ to a basis in $\mathring{\SS}_{2\times 2}$. Therefore the map $\LL_{(\nb \hat{\xi}_{I_1}, \nb \hat{\xi}_{I_2})}$ is an invertible linear map and by \eqref{eq: M_0} 
\begin{equation}
	\label{eq: L inv of (1, 1)}
	\LL^{-1}_{(\nb \hat{\xi}_{I_1}, \nb \hat{\xi}_{I_2})}( M_{[k]}) = (1, 1).
\end{equation}
Now $\LL^{-1}_{(p_1, p_2)}$ is a map that depends smoothly on the parameters $(p_1, p_2)$. 
Thus we may choose a constant $c_2 > 0$ and sufficiently small in the condition \eqref{eq: phase function conditions} such that 
\begin{equation}
	\label{def: c_2}
	|\nb \xi_I - \nb \hat{\xi}_I | < c_2 \implies 
	\left \|
	\LL^{-1}_{(\nb \xi_{I_1}, \nb \xi_{I_2})}
	-
	\LL^{-1}_{(\nb \hat{\xi}_{I_1}, \nb \hat{\xi}_{I_2})} 
	\right \|
	< 
	\frac{1}{4 \| M_{[k]}\|}.
\end{equation}
Furthermore when we define our lifting function $e(t)$ from Section \ref{sec: lifting function} we may choose a sufficiently large constant $K$ such that 
\begin{equation}
	\label{def: K}
	e(t) > K \co{R_\ep} \implies 
	\| \varepsilon\|_{L^\infty} 
	<
	8^{-1} \left \| \LL^{-1}_{(\nb \hat{\xi}_{I_1}, \nb \hat{\xi}_{I_2})} \right \|_{op}^{-1}
	\quad 
	\text{and}
	\quad 
	\| M_{[k]} - \varepsilon\|_{L^\infty} 
	\leq
	2\| M_{[k]}\|_{L^\infty}.
\end{equation}
We combine these two conditions with \eqref{eq: L inv of (1, 1)} to get 
\ALI{
	\Big \|
	(1, 1) 
	-
	\LL^{-1}_{(\nb \xi_{I_1}, \nb \xi_{I_2})}&(M_{[k]} - \varepsilon)
	\Big \|_{L^\infty}
\leq
	\left \|
	\LL^{-1}_{(\nb \hat{\xi}_{I_1}, \nb \hat{\xi}_{I_2})}(M_{[k]})
	-
	\LL^{-1}_{(\nb \hat{\xi}_{I_1}, \nb \hat{\xi}_{I_2})}(M_{[k]} - \varepsilon)
	\right \|_{L^\infty}
	\\&
	+
	\left \|
	\LL^{-1}_{(\nb \hat{\xi}_{I_1}, \nb \hat{\xi}_{I_2})}(M_{[k]} - \varepsilon)
	-
	\LL^{-1}_{(\nb \xi_{I_1}, \nb \xi_{I_2})}(M_{[k]} - \varepsilon) 
	\right \|_{L^\infty}
	\\ &
	\leq 
	\left \| \LL^{-1}_{(\nb \hat{\xi}_{I_1}, \nb \hat{\xi}_{I_2})} \right \|_{op} 
	\left \|
	\varepsilon
	\right \|_{L^\infty}
	+
	\left \|
	\LL^{-1}_{(\nb \xi_{I_1}, \nb \xi_{I_2})}
	-
	\LL^{-1}_{(\nb \hat{\xi}_{I_1}, \nb \hat{\xi}_{I_2})} 
	\right \|_{op}
	\left \|
	M_{[k]} - \varepsilon
	\right \|_{L^\infty}
	\\ &
	\leq 1/2.
}
Let 
$
	(\gamma^2_{I_1}, \gamma^2_{I_2})
	=
	\LL_{(\nb \xi_{I_1}, \nb \xi_{I_2})}(M_{[k]} - \varepsilon)
$
. By our previous calculations we can ensure that  $1/2 < \gamma^2_{I_1}, \gamma^2_{I_2} < 2$ on the time support of $\phi_k(t)$. Therefore there exists $1/2 < \gamma_{I_1}, \gamma_{I_2} < 2$ that solve \eqref{eq: perturbed linear system} pointwise in $(x, t)$. By these calculations we can define the \textbf{coefficients}
\begin{equation} \label{def: gamma}
	\gamma_I = \gamma_I(\nb \xi_I, \varepsilon) \in (2^{-1}, 2)
\end{equation}
as the unique, positive functions solving 
\[
	\sum_{I \in \{k\} \times \F} \theta_I^2(x, t)\hat{K}_{\la, \text{sym}}^{j\ell}[\lambda \nb \xi_I,-\lambda \nb \xi_I] 
	-
	e(t)
	\phi_k^2(t)
	\left (	
	M_{[k]}
	-
	\fr{R_\epsilon^{j\ell}}{e(t)}
	\right )
	= 
	0
\]
from \eqref{eq: R_S sum}. We give $C^0$-norm estimates for $\gamma_I$ in Section \ref{sec: basic estimates}. It follows that \eqref{eq: R_S sum} reduces to 
\begin{equation}
	\label{eq: sum delta B_I}
	R_S = \sum_{I \in \Z \times \F} \delta B_I^{jl}.
\end{equation}
\subsection{Estimates}
\subsubsection{Basic Estimates for the Construction} \label{sec: basic estimates}
We state all of the estimates necessary for controlling the wave corrections. 
\begin{prop}
[Phase Gradient Estimates]\label{prop: Phase gradient estimates}
For the phase gradients $\nb \xi_I$ defined in Section \ref{sec:phase functions}
there exists a positive number, $b_0 \leq 1$ such that the conditions \eqref{eq: phase function conditions} on $\nb \xi_I$ hold for $|t - t(I)| \leq \tau$, 
and such that for all $t \in \R$ with $|t - t(I)| \leq \tau$, we have the following estimates
\ALI{
	\co{ \nb_{\va} D_t^r (\nb \xi_I)} \leqc_{\va} N^{(|\va| + 1 - L)_+/L} \Xi^{|\va|} \nat^r , \quad \text{for all }|\va | \geq 0, r = 0, 1\\
	\co{ \nb_{\va} D_t^2 (\nb \xi_I)} \leqc_{\va} N^{(|\va| + 2 - L)_+/L} \Xi^{|\va|} \nat^2 , \quad \text{for all }|\va | \geq 0, r =  2
}
Moreover the $r = 1$ bound holds also when using the differential operator $\nb_{\va_1} D_t \nb_{\va_2}$ applied to $\nb \xi_I$, where $|\va_1| + |\va_2|\ = |\va|$, and the $r = 2$ bound holds also when using the differential operator $\nb_{\va_1} D_t \nb_{\va_2} D_t \nb_{\va_3}$ applied to $\nb \xi_I$, where $|\va_1| + |\va_2| + |\va_3| = |\va|$. 
\end{prop}
\begin{proof}  See \cite[Sections 17.1-17.3]{isett}.  We note that the choice of $b_0$ depends only on the geometric constant $c_2$ in the first inequality of \eqref{eq: phase function conditions}, which in turn implies the second inequality of \eqref{eq: phase function conditions}.
\end{proof}
\begin{remk}
	Due to $u$ being frequency localized by assumption one can identify $u$ with $u_\epsilon$ from \cite{isett} in this proof. Doing this forgoes the need for any mollification factors like what appears in \cite{isett}, which allows the above estimates 
	to improve to $\co{ \nb_{\va} D_t^r (\nb \xi_I)} \leqc_{\va} \Xi^{|\va|} \nat^r$, for all $|\va | \geq 0, r = 0, 1, 2$. 
\end{remk}
\begin{prop}[Mollification estimates] \label{prop: Mollification estimates}
	We have the following bounds for the $R_\ep$ from Section \ref{sec: Mollification}. 
		\begin{align*}
		\co{\nb_{\va} D_t^r R_\epsilon} 
		&\leqc_{\va}  N^{\fr{3}{2L}(|\va| + r - L)_+}\Xi^{|\va|}\nat^r \D_R, \quad \text{for all }  |\va| \geq 0, r = 0, 1\\
		\co{\nb_{\va} D_t^2 R_\epsilon} 
		&\leqc_{\va}  N^{\fr{3}{2L}(|\va| + 1 - L)_+}\Xi^{|\va|} \nat \epsilon_t^{-1} \D_R, \quad \text{for all }  |\va| \geq 0 \\
		\co{R - R_\ep} &\leqc  \left(
	 \nat \ep_t + \ep_x^L \Xi^{L}
	\right) \D_R  
	\end{align*}
\end{prop}
\begin{proof}
	We cite the proofs of \cite[Proposition 18.5]{isett} and \cite[Proposition 18.7]{isett}. Here we identify the coarse scale flow $u_\ep$ from \cite{isett} with  $u_\ep \coloneqq u = T\theta$ because $u$ is restricted to low frequencies by assumption. Our different choice of mollication factors $\epsilon_x, \epsilon_t$ in \eqref{def: flow mollification parameters} means that the estimates here are the same estimates from \cite{isett} but with $N$ replaced by $N^{\fr{3}{2}}$.  The bound on $\co{R - R_\ep}$ comes from \cite[Section 18.3]{isett}, which employs the decomposition $R - R_\ep = (R - \eta_{\ep_t} \ast_\Phi R) + \eta_{\ep_t}\ast_\Phi (R - R_{\ep_x}) $.  Here $\ast_\Phi$ denotes mollification along the flow as defined in Section~\ref{sec: Mollification} and $R_{\ep_x}$ is the spatial mollification of $R$ in Section~\ref{sec: Mollification}.
\end{proof}
\begin{prop}[Coefficients of the stress equation estimates]\label{prop: gamma estimates}
For components of the amplitude $\varepsilon$ from \eqref{eq: gamma system} and $\gamma_I$ from \eqref{def: gamma} the following bounds hold uniformly for $I \in \Z \times F$
\ALI{
		\co{\nb_{\va} D_t^r \varepsilon} + \co{\nb_{\va} D_t^r \gamma_I } &\leqc_{\va} 
	N^{\fr{3}{2L}(|\va| + 1 - L)_+}\Xi^{|\va|} \nat^r  , &\quad \text{for all } |\va | \geq 0, r = 0, 1\\
		\co{\nb_{\va}D_t^2\varepsilon}  + \co{\nb_{\va}D_t^2\gamma_I} 
		&\leqc_{\va} N^{\fr{3}{2L}(|\va| + 1 - L)_+}\Xi^{|\va|} \nat \epsilon_t^{-1},  &\quad \text{for all } |\va | \geq 0
}
\end{prop}
\begin{proof}
	For the estimates of $\varepsilon_I$ we cite the proof of \cite[Proposition 20.1]{isett} but use our estimates $R_\epsilon$ from Proposition \ref{prop: Mollification estimates}. It is important to add that our choices of $\epsilon_t^{-1}$ from \eqref{def: flow mollification parameters} and $N$ from Lemma \ref{lem: main} are different from \cite{isett} but all of the arguments remain the same. For the bounds on $\gamma_I$ we cite  \cite[Proposition 20.2]{isett}. We remark that in the first bound, slightly sharper estimates hold if the terms are controlled separately but we have combined them together here.
\end{proof}
\begin{prop}[Principal amplitude estimates]
\label{prop: Amplitude estimates}
For the amplitude $\theta_I$ from \eqref{def: theta_I} and the induced velocity vector field $u_I = T\theta_I$ from \ref{sec: velocity corrections} the following bounds hold uniformly for $I \in \Z \times F$
\ALI{
	\co{\nb_{\va} D_t^r \theta_I } + 
	\co{\nb_{\va} D_t^r u_I } 
	&\leqc_{\va} 
	\la^{1/2}N^{\frac{3}{2L}(|\va| + 1 - L)_+ } \Xi^{|\va |} D_R^{1/2}\tau^{-r}, 
	 & \text{for all } |\va | &\geq 0, r = 0, 1 \\
	\co{\nb_{\va} D_t^2 \theta_I } + 
	\co{\nb_{\va} D_t^2 u_I } 
	&\leqc_{\va} 
	\la^{1/2}B_\la^{3/2} N^{\frac{3}{2L}(|\va| + 1 - L)_+ } \Xi^{|\va |}D_R^{1/2}\nat \epsilon_t^{-1} ,
	& \text{for all } |\va | &\geq 0.
}
\end{prop}
\begin{proof}
	First we recall from Section \ref{sec: velocity corrections} that $u_I^l = \th_I m^l(\nb \xi_I )$.  The bounds for $u_I$ follow similarly to the bounds on $\th_I$, so we restrict attention to estimating $\th_I$.  
The following calculations are outlined in \cite[Propositions 21.1 - 21.4]{isett}. They follow the basic scheme that the bounds of $\th_I$ are roughly a rescaling of the estimates for $\gamma_I$. For example we give a calculation of $D_t^2 \th_I$. 
\ALI{
	D_t^2 &\th_I =
	D_t^2 \left [ \la^{1/2}\gamma_I(x, t) \phi \left( \frac{t - k\tau}{\tau} \right )e^{1/2}(t) \right] \\
	&= 
	  \la^{1/2} \gamma_I(x, t) \pr_t^2 \left [ \phi_k(t)e^{1/2}(t) \right] +
	  \la^{1/2}D_t^2 \left [ \gamma_I(x, t) \right] \phi_k(t)e^{1/2}(t) +
	 2 \la^{1/2}D_t\left [ \gamma_I(x, t) \right] \pr_t \left [\phi_k(t)e^{1/2}(t) \right ]
	 \\
	&= 
	A_{I} + A_{II} + A_{III}
}
We begin by estimating the first term using our bounds from Section \ref{sec:Time cutoffs}, \ref{sec: lifting function} and Proposition \ref{prop: gamma estimates}. We will also need the observation that $\Xi e_u^{1/2} \leq \tau^{-1}$ that is due to our definitin of $b^{-1}\geq 1$ in Section \ref{sec:Time cutoffs}. 
\ALI{
	\co{\nb_{\va} A_{I}}
	\leqc  
	\la^{1/2}\co{\nb_{\va} \gamma_I} \tau^{-2}\D_R^{1/2}
	\leqc_{\va}  
	\la^{1/2}N^{\fr{3}{2L}(|\va| + 1 - L)_+}\Xi^{|\va |} \tau^{-2}\D_R^{1/2}
}
We can also bound the second term using the estimates from Proposition \ref{prop: gamma estimates}
\ALI{
	\co{\nb_{\va}A_{II}} 
	\leqc_{\va}  
	\la^{1/2}\co{\nb_{\va}D_t^2 \gamma_I} \D_R^{1/2}
	\leqc_{\va} 
	 \la^{1/2} N^{\fr{3}{2L}(|\va| + 1 - L)_+}\Xi^{|\va |} \nat \epsilon_t^{-1} \D_R^{1/2}
}
For the third term we use a combination of Proposition \ref{prop: gamma estimates} and the definitions from Sections~\ref{sec:Time cutoffs} and \ref{sec: lifting function}. 
\ALI{
	\co{\nb_{\va}A_{III}} 
	\leqc_{\va}  
	\la^{1/2}\tau^{-1} \D_R^{1/2} \co{\nb_{\va} D_t \gamma_I}
	\leqc_{\va} 
	\la^{1/2}N^{\fr{3}{2L}(|\va| + 1 - L)_+}\Xi^{|\va |} \nat \tau^{-1} \D_R^{1/2}
}
Finally, we compare these bounds by proving that $\tau^{-1} \leqc B_\la^{3/4} \epsilon_t^{-1}$ and $\tau^{-2} \leqc B_\la^{3/2} \nat\epsilon_t^{-1}$ with implied constants that only depend on $c_0$ from the definition $\epsilon_t$ given in \ref{sec: Mollification}. We start by proving $\tau^{-1} \leqc B_\la^{3/4}\epsilon_t^{-1}$
\begin{align}\label{eq: inv tau - inv epsilon comparison}
	\tau^{-1} \leqc B_\la^{3/4} \epsilon_t^{-1}
	\iff
	\frac{\D_R^{1/4}B_\la^{3/4}N^{3/4}}{\D_u^{1/4}} \Xi e_u^{1/2} 
	\leqc 
	B_\la^{3/4}
	N^{3/2} \Xi^{3/2}\D_R^{1/2}
	&\iff
	\\
	\nonumber
	B_\la^{3/4}
	\Xi^{3/2}\D_R^{1/4}\D_u^{1/4}N^{3/4}
	\leqc
	B_\la^{3/4}
	N^{3/2} \Xi^{3/2}\D_R^{1/2}
	&\iff
	\frac{\D_u^{1/4}}{\D_R^{1/4}} 
	\leqc
	N^{3/4}
\end{align}
The last inequality holds due to the assumption $N \geq \frac{\D_u}{\D_R} \geq 1$ in the Main Lemma \ref{lem: main}. Next we verify $\tau^{-2} \leqc B_\la^{3/2} \Xi e_u^{1/2}\epsilon_t^{-1}$. 
\ALI{
	\tau^{-2} \leqc B_\la^{3/2} \Xi e_u^{1/2}\epsilon_t^{-1}
	\iff
	\frac{\D_R^{1/2}B_\la^{3/2} N^{3/2}}{\D_u^{1/2}} \Xi^2 e_u
	\leqc
	B_\la^{3/2} 
	 \Xi e_u^{1/2}N^{3/2} \Xi^{3/2}\D_R^{1/2}
	&\iff
	\\
	\D_R^{1/2}N^{3/2} \Xi^3 \D_u^{1/2}
	\leqc
	\Xi^3 \D_u^{1/2}N^{3/2} \D_R^{1/2}
	&\iff
	1 \leqc 1
}
The last inequality holds for a sufficiently small constant, $c_0$, taken in the definition of $\ep_t$ of line \eqref{def: flow mollification parameters}. Now that we have shown $\tau^{-1} \leqc B_\la^{3/4} \epsilon_t^{-1}$ and $\tau^{-2} \leqc B_\la^{3/2} \nat \epsilon_t^{-1}$ we may apply this to compare our previously found bounds of $\nb_{\va} A_I, \nb_{\va} A_{II}$, and $\nb_{\va} A_{III}$. Doing so and taking $B_\la \geq 1 $ so that $B_\la^{3/2}  \geq B_\la^{3/4}$ gives the claimed bound on $\co{\nb_{\va}D_t^2 \th_I}$. 
\end{proof}
\begin{prop}[Bilinear Microlocal correction estimates] 
\label{prop: Bilinear microlocal estimates}
Let $L \geq 2$ and assume the following bounds for the coarse-scale velocity 
\begin{equation}
\label{eq: coarse scale flow assumption}
	\co{\nb_{\va} u} \leqc_{\va} N^{(|\va| -L)_+/L}\Xi^{|\va|} e_u^{1/2} , \quad \text{for all }|\va| \geq 0
\end{equation}
For $\delta B_I^{jl}$ defined in \ref{lem: Bilinear Microlocal} and $D_t = \pr_t + u^i\nb_i$ the following holds uniformly in $I \in \Z \times \F$
\begin{equation}
	\label{eq: delta B estimate}
	\co{ \nb_{\va}D_t^{\rho} \delta B_I^{jl}} \leqc_{\va}
	B_\la^{-1}N^{-1}N^{\fr{3}{2L}(|\va| + 2 - L)_+}\Xi^{|\va|}\D_R \tau^{-\rho},  \quad \text{for all } |\va | \geq 0, 0 \leq \rho \leq 1.
\end{equation}
\end{prop}
We note that the estimate $\ost{\D}_R = \fr{\D_R}{N}$ obtained in the $C^0$ bound above is consistent with an SQG scheme aimed at the optimal regularity.

\begin{proof}
We recall the definition of $\delta B_I$ from Lemma \ref{lem: Bilinear Microlocal}
\ALI{
	\delta B_I^{j\ell}(x, t) &= 
	\int_{\R^2\times \R^2}e^{i\la \nb \xi_I(x) \cdot h_1}e^{-i \la \nb \xi_I(x) \cdot h_2}K^{j\ell}_{\la, \text{sym}}[h_1, h_2]\th_I(x) Y(x, h_1)	dh_1 dh_2 \\
	&+
	\int_{\R^2\times \R^2}e^{i\la \nb \xi_I(x) \cdot h_1}e^{-i \la \nb \xi_I(x) \cdot h_2}K^{j\ell}_{\la, \text{sym}}[h_1, h_2] \th_I(x) Y(x, h_2)  dh_1 dh_2 \\
	&+
	\int_{\R^2\times \R^2}e^{i\la \nb \xi_I(x) \cdot h_1}e^{-i \la \nb \xi_I(x) \cdot h_2}K^{j\ell}_{\la, \text{sym}}[h_1, h_2] Y(x, h_1) Y(x, h_2)  dh_1 dh_2 \\
	&\coloneqq A_I + A_{II} + A_{III}\\
	Y(x, h) &\coloneqq 
	\int_0^1
	\frac{d}{dr} e^{iZ(r, x, h)} \th_I (x - rh)	
	dr\\
	Z(r, x, h) &\coloneqq 
	r\la \int_0^1 h^a h^b \pr_a \pr_b \xi(x - sh ) (1 - s) ds
} 
We show the bound for just $\co{\nb_{\va} D_t^{\rho}A_I}$, with $0 \leq \rho \leq 1$, as the corresponding bounds on $A_{II}$ and $A_{III}$ are similar.  To aid in the computation, we  first decompose $Y(x, h)$ as 
\ALI{
	Y(x, h) &= Y_1(x, h) - Y_2(x, h)\\
	Y_1(x, h) &\coloneqq
	\int_0^1
	e^{iZ(r, x, h)} \left [ i\la \int_0^1 h^a h^b \nb_a \nb_b \xi_I(x - sh)(1-s) ds \right ] 
	\th_I(x - rh)	
	dr\\
	Y_2(x, h) &\coloneqq
	\int_0^1
	e^{iZ(r, x, h)} \nb_a \th_I (x - rh) h^a dr
}
We return our attention to bounding the integral $\nb_{\va}D_t^{\rho }A_I$ in $\nb_{\va}D_t^{\rho}\delta B_I^{jl}$. Consider 
\begin{align}
	A_I &=
	\label{eq: AI1}
	\int_{\R^2\times \R^2} 
	e^{i\la \nb \xi_I(x) \cdot h_1}
	e^{-i\la \nb \xi_I(x) \cdot h_2}
	K^{j\ell}_{\la, \text{sym}}[h_1, h_2]
	\th_I(x) Y_1(x, h_1)dh_1 dh_2 \\
	&-
	\label{eq: AI2}
	\int_{\R^2\times \R^2} 
	e^{i\la \nb \xi_I(x) \cdot h_1}
	e^{-i\la \nb \xi_I(x) \cdot h_2}
	K^{j\ell}_{\la, \text{sym}}[h_1, h_2]
	\th_I(x)Y_2(x, h_1) dh_1 dh_2
\end{align}
For the sake of brevity we will just discuss how to show the claimed bounds on $\nb_{\va} D_t^{\rho} \eqref{eq: AI1}$. By similar methods one can prove the same bounds on $\nb_{\va} D_t^{\rho} \eqref{eq: AI2}$ and by linearity we obtain the bounds on $\nb_{\va}D_t^{\rho} A_I$. Hence we consider just the term $\nb_{\va} D_t^{\rho} \eqref{eq: AI1}$. Using the product rule for $\nb_{\va}D_t^{\rho}$ we may write 
\begin{align}
	\label{eq: Advective derivative of AI1}
	\nb_{\va}D_t^{\rho} \eqref{eq: AI1} &=
	\sum_{ \substack{\va_1 + \dots + \va_6 = \va \\
	\rho_1 + \dots + \rho_6 = \rho}}
	\int_0^1 dr \int_{\R^2\times \R^2} 
	\nb_{\va_1}D_t^{\rho_1} e^{i\la \nb\xi_I(x)\cdot h_1} 
	\nb_{\va_2}D_t^{\rho_2} e^{-i\la \nb\xi_I(x)\cdot h_2}
	K^{j\ell}_{\la, \text{sym}}[h_1, h_2]\\
	\nonumber
	&\cdot
	\nb_{\va_3}D_t^{\rho_3} e^{iZ(r, x, h_1)} 
	\nb_{\va_4}D_t^{\rho_4} iZ(1, x, h_1) 
	\nb_{\va_5}D_t^{\rho_5} \th_I(x - rh_1) 
	\nb_{\va_6}D_t^{\rho_6} \th_I(x) 
	dh_1 dh_2.
\end{align}
We begin with estimates for the first factor in \eqref{eq: Advective derivative of AI1}. We expand this term using the (combinatorial version of the) Fa\' a di Bruno formula as follows
\[
	| \nb_{\va_1} e^{i \la \nb \xi_I(x) \cdot h_1} | \lesssim_{\va_1} \sum_{\pi \in \Pi(|\va_1|)} |h_1|^{|\pi|} \la^{|\pi|} \prod_{B \in \pi} \left\| \nb^{|B|} \nb \xi_I \right\|_{C^0}.
\]
In the line above $\Pi(|\va_1|)$ is the set of all partitions of the set of positive integers $\{1, \ldots, |\va_1|\}$. For a given partition $\pi \in \Pi(|\va_1|)$, we call one subset, $B \in \pi$, a block and we let $|B|$ denote the size of the subset and $|\pi|$ denote the number of blocks in $\pi$. We apply Proposition \ref{prop: Phase gradient estimates} and get
\[
	| \nb_{\va_1} e^{i \la \nb \xi_I(x) \cdot h_1} | \lesssim_{\va_1} \sum_{\pi \in \Pi(|\va_1|)}  |h_1|^{|\pi|} \la^{|\pi|} \prod_{B \in \pi} N^{(|B| + 1-L)_+/L} \Xi^{|B|}.
\]
We now recall an elementary {\it counting inequality}, which states that for any $x_i, M \in \R^+, 1\leq i \leq m$, we have $ \sum_{i = 1}^m ( x_i - M )_+ \leq \left( \left( \sum_{i = 1}^m x_i \right) - M \right)_+$.  A proof can be found in \cite[Lemma 17.1]{isett}.  Applying this inequality with $M = L-1$, and using that $\sum_{B \in \pi} |B| = |\va_1|$ gives
\begin{equation}
	\label{eq: AI1 first factor estimate}
	| \nb_{\va_1} e^{i \la \nb \xi_I(x) \cdot h_1} | \lesssim_{\va_1} (1 + |h|^{|\va_1|}\la^{|\va_1|} ) N^{(|\va_1| + 1-L)_+/L} \Xi^{|\va_1|}.
\end{equation}
Next we treat the advective derivative of the first factor in \eqref{eq: Advective derivative of AI1} using a combination of the Leibniz rule, the Fa\' a di Bruno formula, and our estimates from Proposition \ref{prop: Phase gradient estimates} 
\ALI{
&| \nb_{\va_1} D_t e^{i \la \nb \xi_I(x) \cdot h_1} | \lesssim_{\va_1} \sum_{\vb_1 + \vb_2 = \va_1} \sum_{\pi \in \Pi(|\vb_1|)} |h_1|^{|\pi|} \la^{|\pi|} \left(\prod_{B \in \pi} \left\| \nb^{|B|} \nb \xi_I \right\|_{C^0}\right) \left(|h_1| \la \left\| \nb_{\vb_2} D_t \nb \xi_I \right\|_{C^0} \right)
\\
&\lesssim_{\va_1} \sum_{\vb_1 + \vb_2 = \va_1} \sum_{\pi \in \Pi(|\vb_1|)} |h_1|^{|\pi | +1} \la^{|\pi | + 1} \prod_{B \in \pi} 
\left (
	N^{(|B| + 1  - L)_+/L}\Xi^{|B|}
\right )
\left (
	N^{(|b_2| + 1  - L)_+/L}\Xi^{|b_2|} \nat
\right )
}
In order to simplify this expression we can combine the exponents of the $N$ factors.  We apply the counting inequality with $M = L -1$ and use that $\sum_{B \in \pi} |B| = |\va_1|$ to get
\begin{equation}
	\label{eq: AI1 first factor advective derivative estimate}
	| \nb_{\va_1} D_t e^{i \la \nb \xi_I(x) \cdot h_1} |
	\lesssim_{\va_1}(|h|\la+ |h|^{|\va_1|+1}\la^{|\va_1|+1}) N^{(|\va_1| + 1 - L)_+/L} \Xi^{|\va_1|} (\Xi e_u^{1/2}).
\end{equation}
Combining estimates \eqref{eq: AI1 first factor estimate} and \eqref{eq: AI1 first factor advective derivative estimate} gives
\begin{equation}
	\label{eq: AI1 1st estimate}
	| \nb_{\va_1} D_t^{\rho_1} e^{i \la \nb \xi_I(x) \cdot h_1} | \lesssim_{\va_1} (1+ |h|^{|\va_1|+\rho_1}\la^{|\va_1|+\rho_1}) N^{(|\va_1| + 1 - L)_+/L} \Xi^{|\va_1|} (\Xi e_u^{1/2})^{\rho_1},
\end{equation}
and the same bound holds for the second factor in \eqref{eq: Advective derivative of AI1}
\begin{equation}
	\label{eq: AI1 2nd estimate}
	| \nb_{\va_2} D_t^{\rho_2} e^{-i \la \nb \xi_I(x) \cdot h_2} | \lesssim_{\va_2} (1+ |h|^{|\va_2|+\rho_2}\la^{|\va_2|+\rho_2}) N^{(|\va_2| + 1 - L)_+/L} \Xi^{|\va_2|} (\Xi e_u^{1/2})^{\rho_2}.
\end{equation}
Next we prove the bounds on fourth factor of \eqref{eq: Advective derivative of AI1}, $\nb_{\va} D_t^{\rho_4} Z(r,x, h_1)$, by direct calculation.  For the bound on the spatial derivative, we use Proposition~\ref{prop: Phase gradient estimates} and $\Xi = B_\la^{-1} N^{-1} \la$ to obtain
\ALI{
|\nb_{\va} Z(r,x, h_1)| &\leqc \la |h|^2 \left\| \nb_{\va} \nb^2 \xi_I \right\| \leqc B_\la^{-1} N^{-1} \la^2 |h|^2 N^{(|\va| + 2 - L)_+/L} \Xi^{|\va|},
}
where the constant is independent of $r$.  
For the advective derivative bound, 
we need to approximate the value of $u^i(x)$ in $D_t$ by the value of $u^i$ at a nearby point. For example
\[
	 \left ( \pr_t + u^i(x) \nb_i \right ) \nb_a \nb_b \xi_I(x - sh_1)
	 =
	 D_t  \nb_a \nb_b \xi_I(x - sh_1)
	 +
	 \left (u^i(x) - u^i(x - sh_1) \right ) \nb_i \nb_a \nb_b \xi_I(x - sh_1)
\]
We can control the second term with Taylor's theorem and the counting inequality with $M = L - 2$
\ali{
	\left (u^i(x) - u^i(x - sh_1) \right ) &\nb_i \nb_a \nb_b \xi_I(x - sh_1)
	=
	-sh_1^c \int^1_0 \nb_c u^i(x - \si s h_1)d\si \nb_i \nb_a \nb_b \xi_I(x - sh_1) \label{eq:uPtApproxErrGrad} \\
\left|\nb_{\va_4} \eqref{eq:uPtApproxErrGrad} \right| &\leqc_{\va_4} \sum_{|\vec{b}_1| + |\vec{b}_2| = |\va_4|} |h_1| \left[(\Xi e_u^{1/2})\Xi^{|\vec{b}_1|} N^{(|\vec{b}_1| + 1 - L)_+/L} \right] \Xi^{|\vec{b}_2| + 2} N^{(|\vec{b}_2| + 3 - L)_+/L} \notag \\
\left|\nb_{\va_4} \eqref{eq:uPtApproxErrGrad} \right|  &\leqc_{\va_4} |h_1| \Xi^{|\va_4| + 2} N^{(|\va_4| + 3 - L)_+/L} (\Xi e_u^{1/2}). \notag
}
We use 
$
	N^{(|\va_4| + 3 - L)_+/L} \Xi \leq N^{(|\va_4| + 2 - L)_+/L} \la 
$ 
and $\Xi^{|\va_4|+1} = B_\la^{-1} N^{-1} \la \Xi^{|\va_4|}$ to obtain
\begin{align}
| \nb_{\va_4} D_t^{\rho_4} Z(r,x,h_1)| &\lesssim_{\va_4} \la |h_1|^2 \Xi^{|\va_4| + 1}  (\Xi e_u^{1/2})^{\rho_4} \cdot \left[ N^{(|\va_4| + 2 - L)_+} + {\bf 1}_{\{\rho_4 = 1\}} |h_1| N^{(|\va_4| + 3 - L)_+/L} \Xi \right] 
\nonumber \\
\label{eq: AI1 4th estimate}
&\lesssim_{\va_4}  B_\la^{-1} N^{-1} (|h_1|^2 \la^2 + |h_1|^{2+\rho_4} \la^{2 + \rho_4}) N^{(|\va_4| + 2 - L)_+/L} \Xi^{|\va_4|} (\Xi e_u^{1/2})^{\rho_4}.
\end{align}
Next we compute a bound on the third factor, $\nb_{\va_3} D_t^{\rho_3} e^{i Z(r,x, h_1) }$, as follows 
\begin{align}
\left| \nb_{\va_3} e^{i Z(r,x,h_1)} \right| &\lesssim_{\va_3} \sum_{\pi \in \Pi(|\va_3|)} \la^{|\pi|} |h_1|^{2|\pi|} N^{(|\va_3| + 2 - L)_+/L} \Xi^{|\va_3| +|\pi|}
\nonumber \\
&\lesssim_{\va_3} (1 + \la^{2|\va_3|} |h|^{2|\va_3|}) N^{(|\va_3| + 2 - L)_+/L} \Xi^{|\va_3|} 
\nonumber \\
\left| \nb_{\va_3} D_t e^{i Z(r,x,h_1)} \right| &\lesssim_{\va_3} \sum_{\vb_1 + \vb_2 = \va_3} \left| \nb_{\vb_1}e^{i Z(r,x,h_1)} \nb_{\vb_2}D_tZ(r,x,h_1) ]  \right| 
\nonumber \\
\label{eq: AI1 third factor advective derivative estimate}
&\lesssim_{\va_3} \sum_{\vb_1 + \vb_2 = \va_3} \sum_{\pi \in \Pi(|\vb_1|)} \left(\prod_{B \in \pi} \left| \nb^{|B|} Z(r,x,h_1) \right| \right) \left| \nb_{\vb_2}D_tZ(r,x,h_1) \right|.
\end{align}
If we apply this calculation, Proposition \ref{prop: Phase gradient estimates}, the assumed bounds\eqref{eq: coarse scale flow assumption}, and the counting inequality with $M = L - 2$ to \eqref{eq: AI1 third factor advective derivative estimate} then we get
\[
	\left| \nb_{\va_3} D_t e^{i Z(r,x,h_1)} \right| \lesssim_{\va_3} (|h|^2 \la^2 + |h|^{2|\va_3|+2}\la^{2|\va_3|+2}) N^{(|\va_3| + 2 - L)_+/L} \Xi^{|\va_3|} (\Xi e_u^{1/2}),
\]
and then
\begin{equation}
\label{eq: AI1 3rd estimate}
	\left| \nb_{\va_3} D_t^{\rho_3} e^{i Z(r,x,h_1)} \right| \lesssim_{\va_3} (1 + |h|^{2|\va_3|+2}\la^{2|\va_3|+2})N^{(|\va_3| + 2 - L)_+/L} \Xi^{|\va_3|} (\Xi e_u^{1/2})^{\rho_3}.
\end{equation}
For the fifth factor in \eqref{eq: Advective derivative of AI1} we again approximate $u^i(x)$ in $D_t$ by the value $u^i(x - rh)$ at a nearby point.  
We use Proposition \ref{prop: Phase gradient estimates}, Proposition \ref{prop: Amplitude estimates}, \eqref{eq: coarse scale flow assumption}, and the counting inequality with $M = L - 2$ to get
\begin{align}
\left| \nb_{\va_5} (\pr_t + u^i(x) \nb_i)^{\rho_5} \th_I(x - r h_1) \right| &\leq \left\| \nb_{\va_5} D_t^{\rho_5} \th_I \right\|_{C^0} + \left| \nb_{\va_5} [((u^i(x) - u^i(x - r h_1)) \nb_i )^{\rho_5} \th_I(x - r h_1)]\right| 
\nonumber \\
\label{eq: AI1 5th estimate}
&\lesssim_{\va_5} \Xi^{|\va_5|} N^{\fr{3}{2L}(|\va_5| + \rho_5 + 1 - L)_+} (1 + |h| \la) \la^{1/2} \D_R^{1/2} \tau^{-\rho_5}.
\end{align}
The last factor in \eqref{eq: Advective derivative of AI1}, $\nb_{\va_6}D_t^{\rho_6} \th_I$, can be controlled by a direct application of Proposition \ref{prop: Amplitude estimates}. We can combine this bound with our previously found estimates \eqref{eq: AI1 1st estimate}, \eqref{eq: AI1 2nd estimate}, \eqref{eq: AI1 3rd estimate}, \eqref{eq: AI1 4th estimate} and \eqref{eq: AI1 5th estimate}. Then applying the counting inequality with $M = L - 2$ with the Kernel scaling estimate from \eqref{eq: Kernel scaling} yields
\ALI{
\| \eqref{eq: Advective derivative of AI1} \|_{C^0} &\lesssim_{\va} \sum_{\substack{ \va_1 + \cdots + \va_6 = \va\\ \rho_1 + \cdots + \rho_6 = 1}} \int_0^1 dr \int_{\R^2 \times \R^2} dh_1 dh_2 |K_{\la, \mathrm{sym}}^{j\ell}[h_1,h_2]| \\
&\cdot (1+ |h|^{|\va_1|+\rho_1}\la^{|\va_1|+\rho_1}) N^{(|\va_1| + 1 - L)_+/L} \Xi^{|\va_1|} (\Xi e_u^{1/2})^{\rho_1} \\
&\cdot (1+ |h|^{|\va_2|+\rho_2}\la^{|\va_2|+\rho_2}) N^{(|\va_2| + 1 - L)_+/L} \Xi^{|\va_2|} (\Xi e_u^{1/2})^{\rho_2} \\
&\cdot (1 + |h|^{2|\va_3|+2}\la^{2|\va_3|+2}) N^{(|\va_3| + 2 - L)_+/L} \Xi^{|\va_3|} (\Xi e_u^{1/2})^{\rho_3} \\
&\cdot(|h|^2 \la^2 + |h|^{2+\rho_4} \la^{2 + \rho_4}) N^{(|\va_4| + 2 - L)_+/L} \Xi^{|\va_4|} (\Xi e_u^{1/2})^{\rho_4}  B_\la^{-1} N^{-1}  \\
&\cdot (1 + |h| \la) N^{\fr{3}{2L}(|\va_5| + \rho_5 + 1 - L)_+} \Xi^{|\va_5|} \la^{1/2} \D_R^{1/2} \tau^{-\rho_5} \\ 
&\cdot (1 + |h| \la) N^{\fr{3}{2L}(|\va_6| + \rho_6 + 1 - L)_+} \Xi^{|\va_6|} \la^{1/2} \D_R^{1/2} \tau^{-\rho_6} \\
\| \eqref{eq: Advective derivative of AI1} \|_{C^0} &\lesssim B_\la^{-1} N^{-1} \Xi^{|\va|} N^{\fr{3}{2L}(|\va| + 2 - L)_+} \D_R \tau^{-\rho},
}
where in the last line we used $(\Xi e_u^{1/2}) \leq \tau^{-1}$. This bound completes our work for estimating one summand of $\nb_{\va} D_t^{\rho} A_I$.  The other summand \eqref{eq: AI2} may be handled using the same approach, but in this case the factor $B_\la^{-1} N^{-1}$ arises from the estimates for $h^a \nb_a\th_I(x - rh)$.  By linearity and our estimates of \eqref{eq: AI1} and \eqref{eq: AI2}, we get the claimed bound \eqref{eq: delta B estimate} for $\nb_{\va} D_t^{\rho} A_I$. The remaining integrals $A_{II}$ and $A_{III}$ may be done by an identical calculation. Hence by linearity we have our claimed bound for $\nb_{\va}\delta B_I$.
\end{proof}
\begin{prop}[Microlocal correction estimates]	
	\label{prop: Microlocal correction estimates}
	The following bounds hold uniformly for $I \in \Z \times F$
	\ALI{
		\co{\nb_{\va} D_t^r \delta \th_I} + \co{\nb_{\va} D_t^r \delta u_I} 
		&\leqc_{\va}
		\la^{1/2}B_\la^{-1} N^{-1} N^{\fr{3}{2L}(|\va| + 2 - L)_+}\Xi^{|\va |}\D_R^{1/2}\tau^{-r}
		\quad \text{for all } |\va | \geq 0, r = 0, 1 \\
		\co{\nb_{\va} D_t^2 \delta \th_I} + \co{\nb_{\va} D_t^2 \delta u_I} 
		&\leqc_{\va}
		\la^{1/2}B_\la^{-1} N^{-1} N^{\fr{3}{2L}(|\va| + 2 - L)_+} \Xi^{|\va |}\D_R^{1/2}\nat \epsilon_t^{-1}
		\quad \text{for all } |\va | \geq 0
	}
\end{prop}
\begin{proof}  
	For the same reasons as those given in Proposition \ref{prop: Amplitude estimates} the calculations for $\delta u_I$ are comparable to those for $\delta \th_I$. Hence it will be enough to discuss the estimates for $\delta \th_I$. These calculations are done via direct calculation (see \cite[Lemma 7.5]{isettVicol}, which uses the same microlocal lemma and an analogous scalar correction).  Hence all of our estimates follow from our previous estimates on $\nb \xi_I$ from Proposition~\ref{prop: Phase gradient estimates} and $\th_I$ from Proposition~\ref{prop: Amplitude estimates}.  
\end{proof}
\begin{prop}[Correction Estimates]
	\label{prop: Correction estimates}
	Let $U_I = T \Theta_I$. The following bounds hold uniformly for $I \in \Z \times F$
\ALI{
	\co{\nb_{\va}D_t^r \Theta_I} +
	\co{\nb_{\va}D_t^r U_I } 
	&\leqc_{|\va|} 
	\la^{1/2 + |\va|}\D_R^{1/2}\tau^{-r}
	\quad \text{for all } |\va | \geq 0, r = 0, 1 \\
	\co{\nb_{\va}D_t^2 \Theta_I } +
	\co{\nb_{\va}D_t^2 U_I }
	&\leqc_{|\va|}
	B_\la^{3/2}
	\la^{1/2 + |\va|}\D_R^{1/2}\nat \epsilon_t^{-1}
	\quad \text{for all } |\va | \geq 0
}
\end{prop}
\begin{proof}
Consider $\Th_I = e^{i \la \xi_I}(\th_I + \de \th_I) \coloneqq e^{i\la \xi_I} \tilde{\th}_I$.  Observe by Propositions~\ref{prop: Amplitude estimates} and \ref{prop: Microlocal correction estimates} that $\tilde{\th}_I$ satisfies the same estimates as $\th_I$ (but with larger constants) since the bounds for $\de \th_I$ are lower order.  The proposition then follows from the formulas
\ALI{
D_t \Th_I = e^{i \la \xi_I} D_t \tilde{\th}_I, \quad D_t^2 \Th_I = e^{i \la \xi_I} D_t^2 \tilde{\th}_I 
}
by appealing to the bounds from Propositions \ref{prop: Phase gradient estimates}, \ref{prop: Amplitude estimates}, \ref{prop: Microlocal correction estimates}.  The bounds for $U_I$ follow similarly.
\end{proof}

\subsubsection{Weighted Norm}
For the purpose of a shorter presentation we introduce the following weighted norm that was used in Lemma 4.1 and Lemma 4.2 of \cite{isett2017nonuniqueness}. 
\begin{lem}[Weighted Norm]\label{lem: weighted norm}
Let $D_t = \pr_t + T^l \th \nb_l$ be the advective derivative with respect to $\th$, $\ost D_t = \pr_t + T^l \ost \th \nb_l$ be the advective derivative with respect to $\ost \th = \th + \Theta$, $\la = B_\la N \Xi$ as defined in Section \ref{sec: corrections}, and define $\overset{\diamond}{e}_u{}^{1/2} \coloneqq (\la \D_R)^{1/2}$. For each advective derivative $\mathring{D}_t \in \{ D_t, \ost D_t\}$ define the weighted norm $\mathring{H} \in \{ H, \ost H\}$ by
	\[	
		\mathring H[F] \coloneqq \max_{0\leq r \leq 1} \; \max_{0 \leq |\va| + r \leq L} \fr{ \co{\nb_{\va}  \mathring {D}_t^r F}}{\la^{|\va|}(\la \overset{\diamond}{e}_u{}^{1/2})^r }
	\]	
	 Furthermore $\mathring H[F]$ satisfies the triangle inequality $\mathring H[F  + G] \leq \mathring H[F] + \mathring H[G]$ and the product rule $\mathring H[FG] \leq \mathring H[F] \cdot \mathring H[G]$. The weighted norms are also comparable $\ost H[F] \leqc H[F] \leqc \ost H[F]$. 
\end{lem}
\begin{proof}
	The proof is identical to the proof of \cite[Lemma 4.2]{isett2017nonuniqueness} but with $T^l[\Theta ] = U^l$ identified with $\widetilde{V}^l$, $T^l[\theta] = u^l$ identified with both $v^l$ and $v_\epsilon^l$, and $\overset{\diamond}{e}_u{}^{1/2}$ identified with $e_\phi^{1/2}$.
\end{proof}
\subsubsection{The Second Order Anti-Divergence Operator Bound}
The Lemma below is needed for applying the order -2, second order anti-divergence operator from Section \ref{sec:Second order anti-div} to frequency localized scalar functions. This estimate is from \cite[Lemma 4.3]{isett2017nonuniqueness}. 

\begin{lem} \label{lem:second order anti-div} Suppose $Q \colon \R^2 \to \R$ is Schwartz and satisfies the following bounds
\ALI{
\max_{0 \leq |\va| \leq |\vb| \leq L} \la^{-(|\vb| - |\va|)} \| h^{\va} \nb_{\vb} Q] \|_{L^1(\R^2)} &\leq \La^{-1}
}
for some real number $\La^{-1} \geq 0$. Let $H$ be the weighted norm, $\overset{\diamond}{e}_u{}^{1/2}$ be the energy quantity, and $D_t$ be the advective derivative as defined in Lemma \ref{lem: weighted norm}. Then for any smooth $U$ on $\T^3$ we have
\ALI{ 
	H[Q \ast U] &\leqc \La^{-1} \left( \co{U} + (\la \overset{\diamond}{e}_u{}^{1/2})^{-1} \co{D_t U} \right) 
}
\end{lem}
\begin{proof}
	The proof is identical to \cite[Lemma 4.3 ]{isett2017nonuniqueness}. We present a table where terms used in the proof of the \cite[Lemma 4.3 ]{isett2017nonuniqueness} are listed in the top row and must be replaced by the corresponding term in the bottom row for doing the analogous proof here. 
\[
\left [
\begin{array}{ c | c | c | c | c }
	\bar{H}[\cdot] & \bar{D}_t & e_v  & e_\varphi & N \geq (e_v / e_\varphi)^{1/2}\\
	\hline 
	H[\cdot] & D_t & e_u & \overset{\diamond}{e}_u & N \geq ( e_u / \overset{\diamond}{e}_u )^{1/2} \\
\end{array}
\right ]
\]
The last column of the table is the conditions $N \geq \frac{e_v^{1/2} }{e_\varphi^{1/2}}$ that appears in \cite[Lemma 4.3 ]{isett2017nonuniqueness} and must replaced by the conditions $ N \geq \frac{e_u^{1/2} }{\overset{\diamond}{e}_u{}^{1/2} }$ for the proof here. The condition $N \geq \frac{e_u^{1/2} }{\overset{\diamond}{e}_u{}^{1/2} }$ follows from the assumption $N \geq \D_u / \D_R $ made in the Main Lemma \ref{lem: main} and by taking $B_\la \geq 1$ in the definition $\la = B_\la N \Xi$. 
\end{proof}
\begin{cor}
	\label{cor: H norm}
	Given the second order anti-divergence operator $\RR^{jl}: C^0(\T^2 ) \rightarrow C^0(\T^2)$ defined in \eqref{def: second order anti-div operator}, a frequency projection operator $P_\la$ to frequencies $\{ \xi \subseteq \hat{\R}^2 : \la/10 \leq |\xi| \leq 10 \la \}$, and any smooth $U$ on $\R^2$ we have 
\[
	H[\RR^{jl}P_\la [U]] \leqc \la^{-2} \left( \co{U} +  (\la \overset{\diamond}{e}_u{}^{1/2} )^{-1}\co{ D_t U}\right)
\]
\end{cor}
\begin{proof}
	The result follows immediately from \ref{lem:second order anti-div} by taking $\La^{-1} = C \la^{-2}$. We can define $\La^{-1}$ in this way due to the frequency support of $\RR^{jl}$. 
\end{proof}
\subsubsection{Estimating the Stress Errors}
Recall from \eqref{def: New R} that the new stress $\ost R$ will be decomposed into four terms. We begin with estimating the term $R_T$. 
\subsubsection{Transport Stress Estimate}
\begin{prop}[Transport Stress Error]
	\label{prop: R_T estimate}
	There exists a symmetric tensor, $R_T^{jl}$, as defined in \eqref{def: R_T} that satisfies the estimate
	\[
		H[R_T] \leqc (B_\la N \Xi)^{-3/2} \D_R^{1/2}\nat b^{-1} 
	\] 
\end{prop}
\begin{proof}
Recall from the discussion in Section \ref{sec: Error terms} that 
\begin{equation}
	\label{eq: double div R_T}
	\nb_j \nb_l R_T^{jl} = D_t \Theta + T^l[\Theta]\nb_l \th
\end{equation}
Since the right hand side of this identity has mean-zero on $\T^2$ there is a well defined stress term 
\[
	R_T \coloneqq \RR^{jl} \left [ \pr_t \Theta + T^l \theta \nb_l \Theta + T^l \Theta \nb_l \theta \right  ]
\]
as stated in Definition \eqref{def: R_T}. The frequency support of $\theta$ and $u$ is $\supp \hat{\theta} \cup \supp \hat{u} \subseteq B_{\Xi}(0)$ whereas $\Theta$ has frequency support in an annulus of size $\la$ as explained in Section \ref{sec:freq localizing op}.  Therefore the argument of $\RR$ above has frequency support in $\{ \xi \in \hat{\R}^2 : \la/3 \leq |\xi| \leq 3\la \}$ and we may replace $\RR$ by $\RR P_{\approx \la}$ for the appropriate frequency-localization operator $P_{\approx \la}$.   
Corollary \ref{cor: H norm} then gives
\[
	H[R_T] \leqc \la^{-2}\left( \co{ \eqref{eq: double div R_T} } + (\la \overset{\diamond}{e}_u{}^{1/2})^{-1}\co{D_t \eqref{eq: double div R_T}} \right)
\]
The term $\co{\eqref{eq: double div R_T}}$ equals $\co{D_t \Theta + T^l[\Theta]\nb_l \th}$. By Proposition \ref{prop: Correction estimates} and Definition \ref{def: freq and energy levels}
\begin{equation}
	\label{eq: R_T estimate}
	\co{\eqref{eq: double div R_T} } \leqc
	(\la \D_R)^{1/2} \tau^{-1} + (\la \D_R)^{1/2} \nat 
	\leqc
	(\la \D_R)^{1/2} \tau^{-1} 
\end{equation}
Additionally, we can estimate the term $\co{D_t \eqref{eq: double div R_T}}$ as
\begin{equation}
	\label{eq: D_t R_T estimate}
	\co{D_t \eqref{eq: double div R_T}} 
	= 
	\co{D^2_t \Theta + D_tU^l \nb_l \th + U^l D_t\nb_l \th}
\end{equation}
We can control this quantity using Proposition \ref{prop: Correction estimates}, Definition \ref{def: freq and energy levels}, and a bound for $D_t \nb_l \th$ that follows from the SQG-Reynold's equation $\theta$. From the SQG-Reynold's equation we have
\[
	D_t \nb_l \th 
	=
	\nb_l \nb_a \nb_b R^{ab} - (\nb_l u^k )(\nb_k \th)
\]
and by Definition \ref{def: freq and energy levels}
\[
	\co{D_t \nb_l \th } 
	\leqc
	\Xi^3 \D_R 
	+ 
	\nat^2
	\leqc
	\nat^2
\] 
Using this bound with Proposition \ref{prop: Correction estimates} and Definition \ref{def: freq and energy levels} we can see that 
\[
	\co{ \eqref{eq: D_t R_T estimate} }
	\leqc
	B_\la^{3/2}(\la \D_R)^{1/2}\nat \epsilon_t^{-1} + (\la \D_R)^{1/2} \tau^{-1}\nat 
	+ 
	N^{\fr{3}{2L}(1 -L)_+} (\la \D_R)^{1/2} \nat^2
\]
By our definitions of $\epsilon_t, \tau$ stated in \eqref{def: flow mollification parameters} and Section \ref{sec:Time cutoffs} respectively it follows that
\[
	\nat^2 \leq \nat \tau^{-1} \leq B_\la^{3/2} \nat \epsilon_t^{-1}
\]
where $\tau^{-1} \leq \ep_t^{-1}$ follows from \eqref{eq: inv tau - inv epsilon comparison}. Next we show that $B_\la^{3/2} \nat \epsilon_t^{-1} \leqc B_\la^{-3/4}\la \overset{\diamond}{e}_u{}^{1/2} \tau^{-1}$. 
\ALI{
	B_\la^{3/2} \nat \epsilon_t^{-1} 
	\leqc 
	B_\la^{-3/4}\la \overset{\diamond}{e}_u{}^{1/2}
	&\iff
	B_\la^{3/2} \nat N^{\fr{3}{2L}} \Xi^{3/2}\D_R^{1/2} 
	\leqc 
	B_\la^{-3/4}
	(B_\la N\Xi)^{3/2}\D_R^{1/2} \nat b^{-1}
	\\
	&\iff
	1
	\leqc 
	B_\la^{-3/4}
	\left(
	\frac{\D_R^{1/2}B_\la^{3/2}N^{3/2}}{\D_u^{1/2}}
	\right)^{1/2}
	\iff
	\frac{\D_u^{1/4}}{\D_R^{1/4}}
	\leqc 
	N^{3/4}
}
The last inequality holds due to the assumption $N \geq \D_u / \D_R$ in the Main Lemma \ref{lem: main}. In summary, we have shown that
\[
	\co{ \eqref{eq: D_t R_T estimate} }
	\leqc 
	(\la \D_R)^{1/2} B_\la^{-3/4}(\la \overset{\diamond}{e}_u{}^{1/2})\tau^{-1}.
\]
By combining our estimates for \eqref{eq: R_T estimate} and \eqref{eq: D_t R_T estimate} we conclude that 
\[
	H[R_T] \leqc \la^{-2}(\la \D_R)^{1/2} \tau^{-1}  + B_\la^{-3/4}\la^{-2}(\la \D_R)^{1/2} \tau^{-1} \leqc \la^{-3/2} \D_R^{1/2}\nat b^{-1}.
\]	
\end{proof}
\subsubsection{Hi Frequency Interference Stress Estimate}
\begin{prop}\label{prop: R_H estimate}
There exists a symmetric tensor, $R_H^{jl}$, as defined in \eqref{def: R_H} that satisfies the estimate 
\[
	H[R_H] \leqc b\D_R
\] 
\end{prop}
\begin{proof}

Recall that $R_H^{jl} \coloneqq \RR^{jl} \left [ \sum_{J \neq \bar{I}} U^l_I \nb_l \Theta_J + U^l_J \nb_l \Theta_I \right ]$ as stated in Definition \ref{def: R_H}.  The argument of $\RR$ above has frequency support in $\{ \xi \subseteq \hat{\R}^2 : \la/6\leq |\xi| \leq 6\la \}$ due to  the localizations in frequency and angle of the scalar corrections and velocity corrections defined in Section \ref{sec:freq localizing op}.  In particular, $\RR$ may be replaced by $\RR P_{\approx \la}$ so that Corollary~\ref{cor: H norm} applies.

Since $R_H^{jl}$ is given by a finite sum then by linearity it will be enough to show the claimed bound for a single summand. Let a single summand be $R^{jl}_{H(I,J)} \coloneqq \RR^{jl}[ U^l_I \nb_l \Theta_J + U^l_J \nb_l \Theta_I]$.  
By Corollary \ref{cor: H norm} this quantity is bonded by
\begin{align}
	\label{eq: R_H estimate}
	H[ R^{jl}_{H(I,J)}   ] 
	&\leqc \la^{-2}\left( \co{ \eqref{eq: double div R_H} } + (\la \overset{\diamond}{e}_u{}^{1/2})^{-1}\co{D_t \eqref{eq: double div R_H} } \right) \\
	\label{eq: double div R_H}
 	\nb_j \nb _l R^{jl}_{H(I,J)} &= U^l_I \nb_l \Theta_J + U^l_J \nb_l \Theta_I
\end{align}
We begin with estimating $\co{\eqref{eq: double div R_H}}$. Recall from the definitions given in Sections \ref{sec: velocity corrections} and \ref{sec: corrections} respectively $U_I^l = e^{i\la \xi_I}(u_I^l + \delta u_I^l)$ where $u_I^l = \theta_I m^l(\la \nb\xi_I )$, and $\Theta_I = e^{i\la \xi_I}(\theta_I + \delta \theta_I)$. It follows that 
\begin{align}
	\label{eq: R_HIJ principal term}
	U^l_I \nb_l \Theta_J + U^l_J \nb_l \Theta_I &=
	\la e^{i \la (\xi_I + \xi_J)}
	\left [ 
		\nb_l \xi_J  u_I^l \th_J  + \nb_l \xi_I  u_J^l \th_I  
	\right ] \\
	&+
	\label{eq: R_HIJ lot term 1} 
	\la e^{i \la (\xi_I + \xi_J)}
	\left [ 
		\nb_l \xi_J  u_I^l \delta \th_J  + \nb_l \xi_I  u_J^l \delta \th_I  
	\right ] \\
	&+
	\label{eq: R_HIJ lot term 2} 
	\la e^{i \la (\xi_I + \xi_J)}
	\left [ 
		\nb_l \xi_J \delta u_I^l (\th_J + \delta \th_J) + \nb_l \xi_I \delta u_J^l (\th_I + \delta \th_I) 
	\right ] \\
	&+
	\label{eq: R_HIJ lot term 3} 
	e^{i \la (\xi_I + \xi_J)}
	\left [ 
		 (u_I + \delta u_I)^l \nb_l (\th_J + \delta \th_J) + (u_J + \delta u_J)^l \nb_l (\th_I + \delta \th_I) 
	\right ] 
\end{align}
By Propositions \ref{prop: Amplitude estimates} and \ref{prop: Microlocal correction estimates} we see that $\co{\eqref{eq: R_HIJ lot term 1}}, \co{\eqref{eq: R_HIJ lot term 2}} \leqc \la \Xi \D_R$ and $\co{\eqref{eq: R_HIJ lot term 3}} \leqc \la \Xi \D_R$. We will see that these terms are lower order in $\la$ relative to the leading order term in $\la$, which is $\co{ \eqref{eq: R_HIJ principal term}}$. We consider the term \eqref{eq: R_HIJ principal term} and we calculate the amplitude of this wave, $\nb_l \xi_J  u_I^l \th_J  + \nb_l \xi_I  u_J^l \th_I$, by taking $m^l(\xi) = i\varepsilon^{la} \xi_a /|\xi|$ for $\xi \in \hat{\R}^2$ from Definition \ref{def: SQG}.
\begin{equation}
	\label{eq: R_HIJ principal term amplitude}
	 \nb_l \xi_J  u_I^l \th_J  + \nb_l \xi_I  u_J^l \th_I 
	 =
	 i\th_I\th_J
	 \epsilon^{la} 
	 \left (
	 	\nb_a\xi_I |\nb \xi_I|^{-1} \nb_l \xi_J + \nb_a \xi_J |\nb \xi_J|^{-1}\nb_l \xi_I
	 \right )	
\end{equation}
By the anti-symmetric properites of $\ep^{li}$ we see that $\ep^{la} (\nb_a \xi_I \nb_l \xi_J + \nb_l \xi_I \nb_a \xi_J) = 0$ and so we can add a scalar multiple of this identity into the previous line to obtain
\begin{equation}
	\label{eq: R_HIJ prin. term amp. adjusted}
	\eqref{eq: R_HIJ principal term amplitude}
	 =
	 i\th_I\th_J
	 \epsilon^{la} 
	 \left (
	 	\nb_a\xi_I \nb_l \xi_J (|\nb \xi_I|^{-1}  - |\nb \hat{\xi}_I |^{-1}) + \nb_a \xi_J \nb_l \xi_I (|\nb \xi_J|^{-1} - |\nb \hat{\xi}_J |^{-1})
	 \right )	
\end{equation}
In \eqref{eq: R_HIJ prin. term amp. adjusted} we have introduced $\hat{\xi}_I, \hat{\xi}_J$ that are the initial conditions to $\xi_I, \xi_J$ respecitively as defined in Section \ref{sec:phase functions}. Furthermore $|\nb \hat{\xi}_I |^{-1} = |\nb \hat{\xi}_J |^{-1} = 5^{-1/2}$. In order to control \eqref{eq: R_HIJ prin. term amp. adjusted} we will use $\left | | \nb \xi_I|^{-1} -  |\nb \hat{\xi}_I |^{-1} \right |, \left | | \nb \xi_I|^{-1} -  |\nb \hat{\xi}_J |^{-1} \right | \leq b$. This fact comes from a direct calculation of $\xi_I$ that also holds for $\xi_J$. To start we show that $\co{D_t( |\nb \xi_I|^{-1})} \leqc \nat$. 
\begin{equation}\label{eq: D_t inv phase grad}
	D_t( |\nb \xi_I|^{-1}) = 
	- |\nb \xi_I |^{-3} \nb^a \xi_I D_t \nb_a \xi_I
	\implies 
	\co{D_t( |\nb \xi_I|^{-1})}
	\leqc 
	c_3^{-2} \nat
\end{equation}
Where the inequality uses Proposition \ref{prop: Phase gradient estimates} and $c_3 > 0$ is a small constant that we fix for condition \eqref{eq: phase function conditions}. Then
\begin{equation}
	\label{eq: inv phase grad bound}
	|\nb \xi_I |^{-1} - |\nb \hat{\xi}_I |^{-1} \leq \int_0^\tau \co{D_t( |\nb \xi_I|^{-1})}ds \leqc \tau \nat = b
\end{equation}
We can apply this bound along with the estimates from Propositions \ref{prop: Phase gradient estimates} and \ref{prop: Amplitude estimates} to \eqref{eq: R_HIJ prin. term amp. adjusted} to conclude that $\co{\eqref{eq: R_HIJ prin. term amp. adjusted}} \leqc \la bD_R$, $\co{\eqref{eq: R_HIJ principal term}} \leqc \la^2 b\D_R$, and finally
\begin{equation}
	\label{eq: R_HIJ estimate}
	\la^{-2}\co{\eqref{eq: double div R_H}} \leqc b\D_R
\end{equation}
We also estimate the term $\co{D_t \eqref{eq: double div R_H}}$ by direct calculation. For the sake of brevity we will only present calculations for estimating the advective derivative of the first term \eqref{eq: R_HIJ principal term}, ad the other terms are similar and lower order. We observe that the advective derivative of \eqref{eq: R_HIJ principal term} is equivalent to 
\begin{align}
	D_t \eqref{eq: R_HIJ principal term} 
	&=
	\label{eq: D_t R_HIJ 1}
	\la
	e^{i\la (\xi_I + \xi_J)}
	\left [
	\left ( D_t \nb_l \xi_J \right ) u^l_I \th_J + \left ( D_t \nb_l \xi_I \right ) u^l_J \th_I 
	\right ]
	\\
	&+
	\label{eq: D_t R_HIJ 2}
	\la
	e^{i\la (\xi_I + \xi_J)}
	\left [
	\nb_l \xi_J \left ( D_t  u^l_I \right )  \th_J + \nb_l\xi_I \left ( D_t u^l_J \right ) \th_I 
	\right ] 
	\\
	&+
	\label{eq: D_t R_HIJ 3}
	\la
	e^{i\la (\xi_I + \xi_J)}
	\left [
	\nb_l\xi_J u^l_I \left ( D_t \th_J \right )   + \nb_l \xi_I u^l_J \left ( D_t \th_I  \right )
	\right ]
\end{align}
We directly estimate each term in $C^0$-norm using Propositions \ref{prop: Phase gradient estimates} and \ref{prop: Amplitude estimates} to obtain 
\[
	\co{D_t \eqref{eq: R_HIJ principal term} } \leqc \la^2 \D_R \left ( \nat + \tau^{-1}\right ) \leqc \la^2 \D_R \tau^{-1}
\] 
All that remains is to verify that $\la^{-2} \left ( \la \ed \right )^{-1} \la^2 \D_R \tau^{-1} \leqc b\D_R$. This is equivalent to checking $\tau^{-1}b^{-1} \leqc \la \ed$ via the calcuation
\[
	\tau^{-1}b^{-1}
	\leqc 
	 \la \ed
	\iff 
	\Xi^{3/2}\D_u^{1/2} \left ( \frac{\D_R^{1/2}B_\la^{3/2}N^{3/2}}{\D_u^{1/2}} \right )
	\leqc
	(B_\la N \Xi)^{3/2}\D_R^{1/2}
	\iff
	1
	\leqc
	1
\]
As a result we have
\[
	\la^{-2} \left( \la \ed \right )^{-1}\co{ D_t \eqref{eq: R_HIJ principal term}} \leqc b\D_R
\]
By similar direct calculations of the other terms $D_t \eqref{eq: R_HIJ lot term 1}, D_t \eqref{eq: R_HIJ lot term 2},  D_t \eqref{eq: R_HIJ lot term 3}$ it can be shown that 
\begin{equation}
	\label{eq: D_t R_HIJ estimate}
	\la^{-2} \left( \la \ed \right )^{-1}	\co{ D_t \eqref{eq: double div R_H}} \leqc b\D_R
\end{equation}
and by combining \eqref{eq: R_HIJ estimate} and \eqref{eq: D_t R_HIJ estimate} with \eqref{eq: R_H estimate} we conclude the claimed bound on $H[R_H]$.

\end{proof}
\subsubsection{Low Frequency Stress Estimate}
\begin{prop}
	\label{prop: R_S estimate}
	There exists a symmetric tensor, $R_S^{jl}$ as defined in \eqref{def: R_S} that satisfies the estimate 
	\[
		H[R_S] \leqc B_\la^{-1}N^{-1}\D_R
	\] 
\end{prop}
\begin{proof}
Recall from Section \ref{sec: Error terms} that the Low Frequency Stress Error is defined as 
\[
R_S^{jl} \coloneqq 
\sum_{I = (k, f) \in \Z \times \F} B_\la^{jl}[\Theta_I,  \Theta_{\overline{I}}] - \phi_k^2(t) 
\left ( 
e(t)M_{[k]} - 
R_\epsilon^{jl}
\right)
= \sum_{I \in \Z \times \F} \delta B_I^{jl}
\]
where the last equality holds due to the calculations of Section \ref{sec: algebra of the cancellation} and is stated in \eqref{eq: sum delta B_I} and where $\delta B_I^{jl}$ is defined in Lemma \ref{lem: Bilinear Microlocal}. Therefore it is enough to compute $H \left[ \sum_{I \in \Z \times \F} \delta B_I^{jl} \right ]$ directly. Due to the partition of unity in time $\{ \phi_k^2(t) \}_{k \in \Z}$ that appears in the definition of $\th_I$ and in the definition of $\delta B_I^{jl}$, $R_S$ is a locally finite sum in time. Hence for $0 \leq |\va | \leq L$ and $r = 0, 1$ 
\[
	\left \| \nb_{\va} D_t^r \sum_{I \in \Z \times \F} \delta B_I^{jl} \right \|_{C^0}
	\leqc 
	\sup_{I \in \Z \times \F} 
	\co{\nb_{\va} D_t^r \delta B_I^{jl} } 
	\leqc_{\va}
	B_\la^{-1}N^{-1}N^{\fr{3}{2L}(|\va| + 1 - L)_+}\Xi^{|\va|}\D_R \tau^{-r}.
\]
In the last inequality we used Proposition \ref{prop: Bilinear microlocal estimates}. We remark that the coarse scale velocity assumptions of \eqref{eq: coarse scale flow assumption} are satisfied by the flow $u = T[\theta]$ with the assumed bounds by the frequency and energy levels in Definition \ref{def: freq and energy levels}. Next we show the inequality $\tau^{-1} \leqc B_\la^{-3/4} \la \overset{\diamond}{e_u}{}^{1/2}$ holds due to the following calculation
\ALI{
	\tau^{-1} 
	\leq 
	B_\la^{-3/4} \la \overset{\diamond}{e}_u{}^{1/2}
	&\iff 
	\Xi^{3/2}\D_u^{1/2} \left ( \frac{\D_R^{1/2}B_\la^{3/2}N^{3/2}}{\D_u^{1/2}} \right )^{1/2} 
	\leqc
	B_\la^{-3/4}(B_\la N \Xi)^{3/2}\D_R^{1/2}
	\\
	&\iff
	\frac{\D_u^{1/4}}{\D_R^{1/4}}
	\leqc
	N^{3/4}.
}
The last inequality holds due to our assumtion on $N \geq \D_u / \D_R$ from the Main Lemma \ref{lem: main}. By using the inequality $N^{\fr{3}{2L}(|\va| + 1 - L)_+}\Xi^{|\va|} \leq \la^{|\va|}$ for $0 \leq |\va| \leq L$ and the inequality $\tau^{-1} \leqc \la \overset{\diamond}{e_u}{}^{1/2}$, we have that 
\[
	H[R_S] = H \left [ \sum_{I \in \Z \times \F} \delta B_I^{jl} \right ] \leqc B_\la^{-1}N^{-1}\D_R
\]
\end{proof}
\subsubsection{Mollification Stress Estimate}
\begin{prop}
	\label{prop: R_M estimate}
	There exists a symmetric tensor, $R_M^{jl}$ as defined in \eqref{def: R_M} and there exists a constant $B_\la \geq 1$ sufficiently large so that the following estimate is satisfied
	\[
		\ost H[R_M] \leq \frac{G\D_R}{1000} ,  \quad G \coloneqq \frac{\D_u^{1/4}}{\D_R^{1/4}N^{3/4}}
	\] 
where $G$ is the constant defined in the Main Lemma \ref{lem: main}. 
\end{prop}
\begin{proof}
We 
follow the calculations outlined in \cite[Section 18.3]{isett}. Recall from Lemma \ref{lem: weighted norm} that $\ost H[R_M] \leq C_0 H[R_M]$ for some positive constant $C_0$ and from \eqref{def: R_M} and Section \ref{sec: Mollification}, $R_M \coloneqq R - R_\ep$.

We begin with the 0th derivative bound from Proposition~\ref{prop: Mollification estimates}:
\[
	C_0
	\co{R_M} = C_0 \co{R - R_\ep}
	\leq C \D_R \left (
	 \nat \ep_t + \ep_x^L \Xi^{L}
	\right )
\]
for some positive constant $C$ independent of $B_\la$. By our choice of mollification parameters in \eqref{def: flow mollification parameters}
\[
	\epsilon_x = c_0N^{-\fr{3}{2L}}\Xi^{-1}, \quad
	\epsilon_t = c_0 N^{-3/2}\Xi^{-3/2}\D_R^{-1/2}
\] 
 and by taking the constant $c_0 > 0$ sufficiently small we obtain the bound
$
	C_0\co{R_M} 
	\leq
	\frac{\D_u^{1/2} \D_R^{1/2}}{1000N^{3/2}}
$
. We observe that
$
	\frac{\D_u^{1/2} \D_R^{1/2}}{1000N^{3/2}} 
	\leq 
	\fr{G \D_R}{1000}
$
if and only if
$
	\frac{\D_u^{1/4}}{ \D_R^{1/4}} 
	\leq 
	N^{3/4}
$
. This second condition holds because of the condition $N \geq \D_u / \D_R$ in the Main Lemma \ref{lem: main} so we are able to conclude that
\begin{equation}
	\label{eq: C^0  R_M estimate}
	C_0\co{R_M} 
	\leq 
	\fr{G \D_R}{1000}.
\end{equation}
Our next step is to show estimates for $C_0\co{\nb_{\va}R_M}$ in the case $1 \leq |\va | \leq L$. In this case we use Proposition \ref{prop: Mollification estimates} and the frequency and energy levels assumed in the Main Lemma \ref{lem: main} to get
\[
	C_0 \co{\nb_{\va}R_M} 
	\leq
	C_0
	\left (
	\co{\nb_{\va}R}
	+
	\co{\nb_{\va}R_\ep}
	\right )
	\leq
	C \Xi^{|\va|}\D_R
\]	
for some positive constant $C$ that is independent of $B_\la$. It remains to check $C \Xi^{|\va|}\D_R \leq \frac{\la^{|\va|} G\D_R}{1000}$. This is equivalent to checking $C\frac{\Xi^{|\va|}}{\la^{|\va|}}\leq \fr{G}{1000}$. By the definition of $\la \geq B_\la N \Xi$ in Section \ref{sec: corrections}, the assumption $|\va | \geq 1$, and by taking a constant $B_\la \geq 1$ suffieciently large we have that 
\[
	C\frac{\Xi^{|\va|}}{\la^{|\va|}}
	\leq
	\frac{C}{B_\la N}
	\leq
	\frac{C}{B_\la} \frac{1}{N}
	\leq
	\frac{1}{1000N}
	\leq
	\fr{G}{1000}
\]
Where in the last inequality we have used the condition on $N \geq 1$ from the Main Lemma \ref{lem: main} and the fact that $\left ( \D_u / \D_R \right )^{1/4} \leq 1$. We have now shown 
\begin{equation}
\label{eq: grad R_M estimate}
	C_0\frac{ \co{\nb_{\va} R_M} }{\la^{|\va|}} 
	\leq 
	\frac{G\D_R}{1000}, 
	\quad 
	\text{for }
	0\leq |\va| \leq L 
\end{equation}
Lastly, we want to show an estimate for $C_0\co{\nb_{\va} D_t R_M}$ for $0 \leq |\va | \leq L -1$. By Proposition \ref{prop: Mollification estimates} and Definition \ref{def: freq and energy levels}
\ALI{
	C_0 \co{\nb_{\va} D_t R_M}
	\leq
	\co{\nb_{\va} D_t R}
	+
	\co{\nb_{\va} D_t R_\ep}
	\leq
	C
	\Xi^{|\va|} \D_R \nat
}
for $0 \leq |\va | \leq L - 1$ and some positive constant $C$ independent of $B_\la$.  
The $N \geq \D_u / \D_R$ assumption 
implies $e_u^{1/2} \leq \ed$. Using this bound, the previous bound, and $\la \geq B_\la N \Xi$ 
shows
\[
	C_0 \frac{\co{\nb_{\va} D_t R_M}}{\la^{|\va|}\left ( \la \ed \right )} 
	\leq
	C
	\frac{
		\Xi^{|\va|}\D_R  \nat
	}
	{\la^{|\va|}\left ( \la \ed \right )} 
	\leq
	C \D_R
	\frac{\Xi}{\la} 
	=
	\frac{C}{B_\la} 
	\frac{\D_R}{N}
\]
We observe that if $B_\la$ is a sufficiently large constant then $\frac{C}{B_\la} \leq \frac{1}{1000}$ and $\frac{\D_R}{N} \leq G\D_R$ because of the condition $N \geq \D_u /\D_R$ from the Main Lemma \ref{lem: main}. Therefore we have the estimate 
\begin{equation}
	\label{eq: grad D_t R_M estimate}
	C_0\frac{\co{\nb_{\va} D_t R_M}}{\la^{|\va|}\left ( \la \ed \right )} 
	\leq
	\frac{G\D_R}{1000}
	, \quad
	\text{for } 0 \leq |\va| \leq L - 1
\end{equation}
By our three estimates \eqref{eq: C^0 R_M estimate}, \eqref{eq: grad R_M estimate}, and \eqref{eq: grad D_t R_M estimate} we can conclude the claimed bound
\ALI{
	\ost H[R_M] \leq C_0 H[R_M] \leq \frac{G\D_R}{1000}
}
\end{proof}
\subsection{Verifying the Conclusions of the Main Lemma} \label{sec: verifying the main lemma}
In this section we verify the conclusions of the Main Lemma \ref{lem: main}. We begin by checking that the new SQG-Reynolds flow $(\ost \th, \ost u, \ost R)$ has frequency and energy levels below $(\ost \Xi, \ost \D_u, \ost \D_R)$. First consider the new Reynold's Stress $\ost R$. By Lemma \ref{lem: weighted norm} $\ost H[\ost R] \leq C H[\ost R]$ for some positive constant $C$ that is independent of $B_\la$. By Propositions \ref{prop: R_H estimate}, \ref{prop: R_T estimate},\ref{prop: R_S estimate}, and \ref{prop: R_M estimate}
\ALI{
	\ost H[\ost R] 
	&\leq
	C
	\left (
	H[R_H] + H[R_T] + H[ R_S ] 
	\right )
	+ \ost H[R_M]
	\\
	&\leq
	C \left (
	b\D_R 
	+
	(B_\la N \Xi)^{-3/2}\D_R^{1/2}\nat b^{-1}
	+ 
	B_\la^{-1} N^{-1} \D_R
	\right )
	+
	\fr{G\D_R}{1000}
}
Recalling the definitions from Section \ref{sec:Time cutoffs} and Main Lemma \ref{lem: main}
\[
	b = 
	\left(
		\frac{\D_u^{1/2}}{\D_R^{1/2}B_\la^{3/2}N^{3/2}}
	\right)^{1/2}b_0	,
	\quad
	G
	=
	\frac{\D_u^{1/4}}{\D_R^{1/4}N^{3/4}}
\]	
and by taking $B_\la \gg 1$ sufficiently large the last sum will be small enough to ensure that
\[
	\ost H[\ost R] \leq \frac{G \D_R}{10}
\]
By the definition of $\ost{H}[\cdot ]$ this shows $\ost R$ satisfies the bounds of Definition~\ref{def: freq and energy levels} for the new frequency-energy levels $(\ost \Xi, \ost \D_u, \ost \D_R)$. Additionally we check that $\ost \th, \ost u$ satisfy the desired bounds.  
We present only the calculations for $\ost \th$ as the $\ost u$ calculations are completely identical. For $0 \leq |\va| \leq L$ and $r = 0, 1$ 
\[
	\co{\nb_{\va} D_t^r \ost \th}
	\leq
	\co{\nb_{\va} D_t^r \th}
	+
	\co{\nb_{\va} D_t^r \Theta}	
	\leqc
	N^{\frac{3}{2L}(|\va| + 1 - L)_+}
	\Xi^{|\va|}
	\D_R^{1/2}\tau^{-r}
	+
	\la^{|\va| + 1/2}\D_R^{1/2}\tau^{-r}
\]
where the last inequalities are due to Propositions \ref{prop: Amplitude estimates}, \ref{prop: Correction estimates}. We observe that for $L \geq 2$ we have $N^{\fr{3}{2L}(|\va| + 1 - L)_+ }\Xi^{|\va|} \leq \la^{|\va|} \leq \ost \Xi^{|\va|} $. Additionally $\la^{1/2}D_R^{1/2} \leq \ost e_u^{1/2}$ and we verify $\tau^{-1} \leq \ost \Xi \ost e_u^{1/2}$ below
\ALI{
	\tau^{-1} \leq \ost \Xi \ost e_u^{1/2}
	\iff
	b^{-1}\Xi^{3/2}\D_u^{1/2} 
	\leq
	(\hc N \Xi)^{3/2}\D_R^{1/2}
	&\iff \\
	b_0^{-1}\frac{\D_u^{1/4}B_\la^{3/4}N^{3/4}}{\D_u^{1/4}}\D_u^{1/2}
	\leq 
	(\hc  N)^{3/2}\D_R^{1/2}
	&\iff
	\frac{\D_u^{1/4}}{\D_R^{1/4}}B_\la^{3/4}
	\leq 
	b_0
	\hc^{3/2}N^{3/4}
}
Here the last inequality holds due to our choice of $N \geq \D_u / \D_R$ in the Main Lemma \ref{lem: main} and by taking $\hc \gg B_\la$. By these inequalities we have that for $0 \leq |\va| \leq L$ and $r = 0, 1$ 
\[
	\co{\nb_{\va} D_t^r \ost \th}
	\leqc
	\la^{|\va| + 1/2}\D_R^{1/2}\tau^{-r}
	\leq
	\ost \Xi^{|\va|} \ost e_u^{1/2} 
	\left (
	\ost \Xi \ost e_u^{1/2} 
	\right )^r,
\]
where again the last inequality holds for a sufficiently large choice of constant $\hc$. We have now shown the claimed frequency and energy levels $(\ost \th, \ost u, \ost R) \leq (\ost \Xi, \ost \D_u, \ost \D_R)$. Next we will show the claimed bounds of $\La^{-1/2}\Theta $ and $W^i$ where $\La^{-1/2}\Theta = \nb_i W^i$. Recall from Section \ref{sec:freq localizing op} that the scalar correction $\Theta$ has a frequency supp in the frequency band of scale $\la$. The compact frequency supports with standard Littlewood-Paley theory and Proposition \ref{prop: Correction estimates} gives the following estimate for $|\va | \leq 1$
\[
	\co{\nb_{\va} \La^{-1/2} \Theta} 
	\leqc \la^{|\va | - 1/2}\co{\Theta}
	\leqc
	\la^{|\va |}D_R^{1/2} 
\]
By taking $\hc \gg B_\la$ sufficiently large we get the claimed estimate from the Main Lemma \ref{lem: main}. 
\[
	\co{\nb_{\va} \La^{-1/2}\Theta} \leq \hc (N \Xi)^{|\va|}\D_R^{1/2}
\]
Next let $\RR^i \coloneqq (-\Delta)^{-1}\nb^i$ be a first order anti-divergence operator and let $P_\la$ be an operator that localizes to frequency $\{ |\xi| \sim \la \}$ such that $\Th = P_\la \Th$.  
Define $W^i \coloneqq \RR^i P_{\la}\La^{-1/2} \Theta$.  Then
\[
	\co{W^i} 
	=
	\co{\RR^i P_{\la}\La^{-1/2} \Theta } 
	\leqc
	\la^{-1}\co{\La^{-1/2} \Theta} 
	\leqc \la^{-1}D_R^{1/2}.
\]
For $|\va| \leq 1$, we then have
\[
	\la \co{\nb_{\va} W^i} 
	+
	\co{\nb_{\va} \La^{-1/2} \Theta} 
	\leqc
	\la^{|\va |}D_R^{1/2}
\]
By using the definition of $\la \geq B_\la N\Xi$ from Section \ref{sec: corrections} and by taking $\hc \gg B_\la$ sufficiently large this shows the claimed estimate from the Main Lemma \ref{lem: main}
\[
	\co{\nb_{\va} W^i} 
	\leq
	\hc 
	(N\Xi )D_R^{1/2}
\]
Next we consider $\co{\pr_t \La^{-1/2}\Theta}$. We begin with the term $\pr_t \La^{-1/2}\Theta$
\begin{equation}
	\label{eq: pr_t Theta in C -1/2}
	\pr_t \La^{-1/2}\Theta
	=
	\La^{-1/2}D_t \Theta - \La^{-1/2} \left ( u^l \nb_l \Theta \right )
\end{equation}
Consider the first term \eqref{eq: pr_t Theta in C -1/2}. We observe that $D_t \Theta$ is also supported in frequency at scales $\la$ due to the how $u = T \th$ and $\hat{\th}$ is supported at scales $\Xi \ll \la$ as stated in Definition \ref{def: freq and energy levels}. By the frequency support and Proposition \ref{prop: Correction estimates}
\[
	\co{\Lambda^{-1/2} D_t \Theta} 
	\leqc
	\la^{-1/2}\co{ D_t \Theta} 
	\leqc
	\la^{-1/2} \la^{1/2}\D_R^{1/2}\tau^{-1}
	=
	\D_R^{1/2}\tau^{-1}
\]
For the second term in \eqref{eq: pr_t Theta in C -1/2} we again use the fact that the function $u^l \nb_l \Theta$ is supported in frequency at scales $\la$ to obtain 
\[
	\co{  \La^{-1/2} \left ( u^l \nb_l \Theta \right )}
	\leqc
	\la^{-1/2} \co{u} \co{\nb_l\Theta}
	\leqc
	e_u^{1/2} \la \D_R^{1/2}.
\]	
In total we have shown 
\[
	\co{\pr_t \La^{-1/2}\Theta} \leqc \D_R^{1/2} \left( \tau^{-1} + \la e_u^{1/2} \right ).
\]
We now compare the time scales and we verify that $\tau^{-1} \leqc B_\la^{-1/4} \la e_u^{1/2}$ with the calculation below:
\ALI{
	\tau^{-1} 
	\leqc
	B_\la^{-1/4}
	\la e_u^{1/2}
	\iff
	b^{-1}
	\leqc 
	B_\la^{3/4} N
	\iff
	b_0^{-1}
	\frac{\D_R^{1/4}B_\la^{3/4} N^{3/4}}{\D_u^{1/4}}
	\leqc 
	B_\la^{3/4} N
	\iff
	\frac{\D_R^{1/4}}{\D_u^{1/4}}
	\leqc
	N^{1/4}.
}
Where the last inequality holds because $\D_R / \D_u \leq 1$ and $N \geq 1$. By this calculation we can obtain the claimed estimate by taking a constant $\hc \gg B_\la$ sufficiently large to get
\[
	\co{\pr_t \La^{-1/2}\Theta} 
	\leqc
	\D_R^{1/2} \la e_u^{1/2}
	\leq
	\hc \D_R^{1/2} (N\Xi e_u^{1/2})
\]
All that remains to check is the time support of condition \eqref{eq: time support growth}. This containment follows from the time support of the lifting function $e(t)$ defined in Section \ref{sec: lifting function} and by verifying that $\ep_t \ll \nat^{-1}$, which in turn follows from the definition of $\ep_t = c_0 N^{-3/2} \Xi^{-3/2}\D_R^{1/2}$ in \eqref{def: flow mollification parameters} and the condition $N \geq \D_u / \D_R$ from the Main Lemma \ref{lem: main}. 
\section{Main Lemma Implies the Main Theorem} \label{sec:mainImpMainThm}
In this section we prove the Main Theorem \ref{lem: main} using the Main Lemma \ref{lem: main}. Before applying the Main Lemma \ref{lem: main} we fix the parameter $L = 2$. Given $0 < \a < 3/10, \; 0 < \b < 1/4$, we will produce by iteration of the Main Lemma a sequence of solutions $(\th_{(k)}, R_{(k)})$ to the SQG-Reynold's system such that $R_{(k)} \rightarrow 0$ uniformly in $C^0(\R \times \T^2)$, the time support of $\th_{(k)}$ is contained in the interval $J_{(k)}$, and the sequence of functions $\{ \La^{-1/2}\th_{(k)}\}_{k \in \N}$ converges in $C_t^\b C^0_x \cap C_t^0 C_x^\a(\R \times \T^2)$. We remark that the non-triviality of our constructed solutions will be a consequence of the h-principle proven in Section \ref{sec: h-principle}. 
\subsection{The base case $k = 0$}
Fix $\a,\b$ be such that $0 < \a < 3/10, 0 < \b < 1/4$.  For stage $k = 0$ we assume that $(\th_{(0)}, R_{(0)})$ are any given smooth, compactly supported SQG-Reynolds flow on $\R \times \T^2$ (possibly $0$).  We assume also that the support of the spatial Fourier transform $\supp \hat{\th}$ is bounded in frequency space, uniformly in time.   

Let $J_{(0)}$ be a nonempty time interval containing the support of $(\th_{(0)}, R_{(0)})$.  We choose a triple of parameters $\left(
		\Xi_{(0)}, \D_{u(0)}, \D_{R(0)}
	\right)$ that are admissible according to Definition~\ref{def: freq and energy levels} of the frequency and energy levels.  Namely, take $\D_{R(0)}$ to be any number $\D_{R(0)} \geq \co{R_{(0)}}$, and set $\D_{u(0)}$ to equal $Z\D_{R(0)}$, where $Z = Z_{\a,\b} \gg 1$ is chosen to satisfy conditions~\eqref{eq: Z condition alpha} and \eqref{eq: Z condition beta} below.  Finally, take $\Xi_{(0)}$ sufficiently large so that the remaining conditions in Definition~\ref{def: freq and energy levels} are satisfied.

\subsection{Choosing the parameters} \label{sec: choosing the parameters}
We iterate the Main Lemma~\ref{lem: main} to produce SQG-Reynolds solutions  $(\th_{(k)}, R_{(k)})$ that are bounded above by the corresponding frequency and energy levels $\left( \Xi_{(k)}, \D_{u(k)}, \D_{R(k)}\right)$. We also impose the ansatz
\begin{equation}
	\label{eq: ansatz}
	\D_{R(k+1)} = \frac{\D_{R(k)}}{Z}, 
	\quad
	\text{for all }k \geq 0.
\end{equation}
The parameter $Z = Z_{\a,\b}$ is chosen sufficiently large to guarantee the condtions \eqref{eq: Z condition alpha}, \eqref{eq: Z condition beta}.

To achieve this ansatz we must choose the frequency growth parameter $N$ to be 
\begin{equation}
	\label{eq: freq increment}
	N = Z^{5/3}.
\end{equation}
Note that by the ansatz, after the $k$th application of the Main Lemma \ref{lem: main}
$
	Z = \frac{\D_{R(k-1)}}{\D_{R(k)}} = 
	\frac{\D_{u(k)}}{\D_{R(k)}}
$
and it follows that after the $k+1$th application of the Main Lemma \ref{lem: main}
\[
	\D_{R(k+1)} = G_{(k)} \D_{R(k)}
	=
	\frac{Z^{1/4}}{N^{3/4}} \D_{R(k)}
	=
	\frac{\D_{R(k)}}{Z}.
\]
Hence by our choice of $N$ we can maintain the ansatz for all $k \geq 0$. 
\subsection{Regularity of the Solution}
In this section we will show the convergence of $\{ \La^{-1/2}\th_{(k)}\}_{k \in \N}$ in $C^\b_tC_x^0 \cap C_t^0C^\a_x(\R \times \T^2)$ for the given $\a, \b$.  As a consequence, we have strong convergence of $\th_{(k)}$ in $L_t^2 H^{-1/2}$, which implies weak convergence of the nonlinearity $\nb_l[ \th_{(k)} T^l[\th]] \rightharpoonup \nb_l[\th T^l[\th]]$ in $\DD'(\R \times \T^2)$.  Since $\nb_j\nb_l R_{(k)}^{jl} \rightharpoonup 0$, the limiting scalar field $\th$ is therefore a weak solution to the SQG equation.

We first calculate the $C_x^\a(\T^2)$ regularity. Define the partial sum
\[
	\La^{-1/2}\th_{(K)}
	\coloneqq
	\sum_{k = 0}^{K}\La^{-1/2} \Theta_{(k)}
\]
where $\Theta_{(k)}$ is the scalar correction from the $k$th application of the Main Lemma \ref{lem: main}. From this defintion it follows that
\[
	\ca{\La^{-1/2}\th_{(K)}}
	\leq 
	\sum_{k = 0}^{K}\ca{\La^{-1/2} \Theta_{(k)}}
\]
and it will enough to show a bound on $\ca{\La^{-1/2}\Theta_{(k+1)}}$ that decays exponentially in $k$. We may interpolate between the estimates for $|\va | = 0, 1$ derivatives given at the end of the $k+1$th application of the Main Lemma \ref{lem: main} to obtain the following bound for $\a \in (0, 1)$
\[
	\ca{\La^{-1/2}\Theta_{(k+1)}} \leq \hc (N \Xi_{(k)})^\a \D_{R(k)}^{1/2}
\]
Define the upperbound quantity
\[
	E_{\a (k+1)} \coloneqq
	(N \Xi_{(k)})^\a \D_{R(k)}^{1/2} 
\]
so that $\ca{\La^{-1/2}\Theta_{(k+1)}} \leq  \hc E_{\a (k+1)}$. We will show that this upperbound $E_{\a (k+1)}$ decays exponentially in $k$ for an appropriate parameter $\a$.  We apply the ansatz \eqref{eq: ansatz}, our choice of $N = Z^{5/3}$ from \eqref{eq: freq increment}, and the frequency and energy levels from the Main Lemma \ref{lem: main}
\[
	E_{\a (k+1)}
	=
	Z^{5\a / 3}(\hc N \Xi_{(k-1)})^\a Z^{-1/2} \D_{R(k-1)}^{1/2}
	=
	\hc^{\a} Z^{\fr{10\a - 3}{6}} 
	(N \Xi_{(k-1)})^\a \D_{R(k-1)}^{1/2}
	=
	\hc^{\a} Z^{\fr{10\a - 3}{6}} 
	E_{\a (k)}
\]
Now we observe that $E_{\a (k+1)} < E_{\a (k)}$ if and only if 
$
	\hc^{\a} Z^{\fr{10\a - 3}{6}} 
	=\left (
		\hc^{\frac{6\a}{10\a - 3}}Z
	\right )^{\frac{10\a - 3}{6}}< 1
$. For a given $0 < \a < 3/10$ we choose the parameter $Z \gg 1$ sufficiently large so that 
\begin{equation}
	\label{eq: Z condition alpha}
	\left (
	\hc^{\frac{6\a}{10\a - 3}}Z
	\right ) < 1.
\end{equation}
It follows that, with $Z$ as above,  
the upperbound quantity $E_{\a(k+1)}$ decays exponentially in $k$. Hence for the function $\La^{-1/2}\th \coloneqq \lim_{K \rightarrow \infty}\La^{-1/2}\th_{(K)}$ 
\[
	\max_t \ca{\La^{-1/2}\th}
	\leq
	\sum_{k = 0}^{\infty} \max_t \ca{\La^{-1/2} \Theta_{(k)}}
	< 
	\infty
\]
for the given $\a < 3/10$.

Next we establish the $C^{\b}_tC_x^0$ regularity. By interpolating the estimates at the end of $k+1$th application of the Main Theorem \ref{lem: main} again we obtain the bound 
 \[
	\left\|\La^{-1/2}\Theta_{(k+1)} \right\|_{C_t^\b C_x^0} \leq \hc \D_{R(k)}^{1/2}\left ( (N\Xi_{(k)}) e_{u(k)}^{1/2} \right )^\b
 \]
and we define the upperbound quantity
\[
	E_{\b(k+1)} \coloneqq
	\D_{R(k)}^{1/2}\left ( N\Xi_{(k)} e_{u(k)}^{1/2} \right )^\b
\]
In order to show that this quantity decays as exponentially in $k$ we will need the identity 
\begin{equation}
	\label{eq: e_u(k) identity}
	e_{u(k)}^{1/2} = \hc^{1/2} Z^{1/3} e_{u(k-1)}^{1/2}.
\end{equation}
This identity can be directly computed using the Main Lemma \ref{lem: main}, our choice of $N = Z^{5/3}$ from \eqref{eq: freq increment}, ansatz \eqref{eq: ansatz}, and the identity $\D_{u(k)} = \frac{\D_{u(k-1)}}{Z}$ that comes from combining the ansatz \eqref{eq: ansatz} with the Main Lemma \ref{lem: main} 
\[
	e_{u(k)}^{1/2} 
	=
	\Xi_{(k)}^{1/2}\D_{u(k)}^{1/2}
	=
	(\hc N\Xi_{(k-1)})^{1/2}\frac{\D_{u(k-1)}^{1/2}}{Z^{1/2}}
	=
	\hc^{1/2} N^{1/2}Z^{-1/2}e_{u(k-1)}^{1/2}
	=
	\hc^{1/2} Z^{1/3} e_{u(k-1)}^{1/2}.
\]
We apply \eqref{eq: e_u(k) identity} and calculate $E_{\b(k+1)}$ as
\ALI{
	E_{\b (k+1)}
	&=
	\frac{\D_{R(k-1)}^{1/2}}{Z^{1/2}} \left ( Z^{5/3} (\hc N\Xi_{(k-1)}) \hc^{1/2} Z^{1/3} e_{u(k-1)}^{1/2} \right )^\b
	\\
	&=
	\hc^{\frac{3\b}{2}}Z^{2\b - \frac 12} \D_{R(k-1)}^{1/2} \left  ( N\Xi_{(k-1)} e_{u(k-1)}^{1/2} \right )^\b
	\\
	&=
	\hc^{\frac{3\b}{2}}Z^{\frac{4\b - 1}{2}} E_{\b (k)}.
}
We observe that the upperbound quantity $E_{\b (k+1)}$ decays exponentially in $k$ if and only if 
\[
	\hc^{\frac{3\b}{2}}Z^{\frac{4\b - 1}{2}} = \left ( \hc^{\frac{3\b}{4\b -1 }}Z \right )^{\frac{4\b - 1}{2}} < 1.
\]
For a given $0 < \b <  1/4$ we can choose a parameter $Z \gg 1$ sufficiently large so that 
\begin{equation}
	\label{eq: Z condition beta}
	\left ( \hc^{\frac{3\b}{4\b -1 }}Z \right )^{\frac{4\b - 1}{2}} < 1.
\end{equation}
From this inequality, it follows that $E_{\b (k+1)}$ decays as exponentially in $k$ and 
\ALI{
	\left\|\La^{-1/2}\th\right\|_{C_t^\b C_x^0}
	\leq
	\sum_{k = 0}^{\infty} \left\|\La^{-1/2} \Theta_{(k)}\right\|_{C_t^\b C_x^0}
	< 
	\infty,
	}
for the given $\b < 1/4$.  

Since $C_t^\b C_x^0$ and $C_t^0 C_x^\a$ are Banach spaces, we have shown that $\La^{-1/2}\th \in C_t^\b C_x^0 \cap C_t^0 C_x^\a$.

\subsection{Compact Support in Time of the Solution}
In this section we will discuss how the limiting SQG solution, $\displaystyle \th \coloneqq \lim_{k\rightarrow \infty} \th_{(k)}$, has compact support in time. By \eqref{eq: time support growth} it is enough to show that the series 
$
	\sum_{k = 0}^\infty \left( \Xi_{(k)} e_{u(k)}^{1/2}\right )^{-1}
$
converges. By the choice of $N = Z^{5/3}$ from \eqref{eq: freq increment} and the identity \eqref{eq: e_u(k) identity}, we have
\[
	\left ( \Xi_{(k)} e_{u(k)}^{1/2} \right )^{-1}
	=
	\hc^{-3/2} Z^{-2}
	\left ( \Xi_{(k-1)} e_{u(k-1)}^{1/2} \right )^{-1}
\]
We note that $\hc^{-3/2} Z^{-2} < 1$, and so
$
	\sum_{k = 0}^\infty \left( \Xi_{(k)} e_{u(k)}^{1/2}\right )^{-1}
	< 
	\infty
$.
\section{An h-principle}\label{sec: h-principle}
We present the proof to Theorem \ref{thm: h-principle} here. 

Let $f$ be a given smooth scalar field with compact support that satisfies $\int_{\T^2}f(x,t) dx = 0$ for all $t$, and let $0 < \a < 3/10, 0 < \b < 1/4$ be given.  We want to construct a sequence of weak solutions $\th_n$ to the SQG equation of class $\La^{-1/2} \th_n \in C_t^0 C_x^\a \cap C_t^\b C_x^0$ with such that $\La^{-1/2} \th_n \rightharpoonup \La^{-1/2} f$ in $L_{t,x}^\infty$ weak-$\ast$.

The first step is to approximate $f$ by scalar fields $f_n$, $n \in \N$, that have compact frequency support.  Let $\eta_n(h) = 2^{2n} \eta(2^n h)$ be a standard mollifying kernel in the spatial variables with compact frequency support $\mbox{supp } \hat{\eta}_n \subseteq \{ |\xi| \leq 2^n \}$.  Set $f_n = \eta_n \ast f$.  With this choice, we have that
\ali{
 \sup_n \| \nb_{\va} \pr_t^r f_n \|_{C^0} \leqc \| \nb_{\va}\pr_t^r f \|_{C^0}, \text{ for } 0 \leq |\va|, r, \text{ and } \lim_{n \to \infty} \co{ \La^{-1/2}(f_n - f) } = 0. \label{eq:boundsOnfn}
}
The time support of $f_n$ is also contained in the time support of $f$.

Now observe that the scalar fields $f_n$ can be realized as SQG-Reynolds flows.  
Namely, consider $\pr_t f_n + T^l f_n\nb_l f_n$ and observe that this function has mean-0 over $\T^2$ due to the non-linear term $T^l f_n\nb_l f_n = \nb_l[f_n T^lf_n]$ being the divergence of a smooth vector field and the assumption  that $\int_{\T^2}f(x,t) dx = 0$ in  Theorem \ref{thm: h-principle}. Therefore we may apply the second-order anti-divergence operator $\RR$ defined in \eqref{def: second order anti-div operator}.  Specifically, we define symmetric and traceless tensors $R_n$ such that
\[
	R_n \coloneqq \RR \left [ \pr_t f_n + T^l f_n\nb_l f_n \right], \text{which implies } 
	\pr_t f_n + T^l f_n\nb_l f_n = \nb_j \nb_l R_n^{jl}.
\]
The pair $(f_n, R_n)$ is thus a smooth SQG-Reynolds flow with frequency support $\mbox{supp } \hat{f} \subseteq \{ |\xi| \leq 2^n \}$.  Observe also that by \eqref{eq:boundsOnfn}, we have a uniform bound on the errors $R_n$:
\ALI{
\sup_n \| R_n \|_{C^0} &\leq \dmd{\D}_R < \infty \quad \text{for some } \dmd{\D}_R > 0.
}

Starting with $(f_n, R_n)$, we now construct a weak solution $\th_n$ to SQG close to $f_n$ in the appropriate weak topology.  
The SQG solution $\th_n$ is constructed by iterating the Main Lemma \ref{lem: main} using an initial choice frequency and energy levels 
$
	\left ( 
		\Xi_{n(0)}, \D_{R, n(0)}, \D_{u, n(0)}
	\right )
$
chosen as follows.  Take $\D_{R,n(0)} = \dmd{\D}_R > 0$ and $\D_{u,n(0)} = Z \D_{R,n(0)}$ where $Z = Z_{\a,\b} \gg 1$ is as in Section~\ref{sec:mainImpMainThm}.  Take $\Xi_{n(0)}$ sufficiently large such that the SQG-Reynolds flow $(\th_{n(0)}, R_{n,(0)}) \coloneqq (f_n, R_n)$ satisfies Definition~\ref{def: freq and energy levels}, and such that
\begin{equation}
	\label{eq: initial freq condition}
	\lim_{n \to \infty} \Xi_{n(0)} = \infty.
\end{equation}
In particular, the frequency support condition will be satisfied for $\Xi_{n(0)} \geq 2^n$.  

Repeated application of the Main Lemma using the parameter choices of Section~\ref{sec:mainImpMainThm} yields a sequence of SQG-Reynolds flows $(\th_{n(k)}, R_{n(k)})$ such that $\th_{n(k)}$ converges as $k \to \infty$ to an SQG solution $\th_n$ with $\La^{-1/2} \th_n \in C_t^\b C_x^0 \cap C_t^0 C_x^\a$.  
We claim that the solutions $\th_n$ satisfy $\La^{-1/2} \th_n \rightharpoonup \La^{-1/2} f$ in $L_{t,x}^\infty$ weak-*.  

Let $\ep > 0 $ and a test function $\vphi \in L^1(\R \times \T^2)$ be given.  We will show that there exists a large index $M$, depending on $\vphi$ and $\ep$, such that for all $n \geq M$
\[	
	\left |
		\int_{\R \times \T^2} \La^{-1/2}(\th_n - f) \vphi dx dt 
	\right |
	< 3\ep.
\]
First, observe by \eqref{eq:boundsOnfn} and H\"{o}lder that for all $n \geq M$ sufficiently large we have 
\ali{
\label{eq:cptFreqSuppApprox}
\left |
		\int_{\R \times \T^2} \La^{-1/2}(f_n - f) \vphi dx dt 
	\right |
	< \ep.
}
Next, take a smooth approximation of compact support $\tilde{\vphi} \in C_c^\infty(\R \times \T^2)$ with
\begin{equation}
	\label{eq: choice of vphi approx}
	\left \|
		\vphi - \tilde{\vphi}
	\right \|_{L^1(\R \times \T^2)}
	< 
	\frac{\ep}{ 2 \hc \dmd{\D}_R^{1/2}},
\end{equation}
where $\hc$ is the constant from the Main Lemma \ref{lem: main}.  It now suffices to estimate the integral 
\begin{align}
	\left |
		\int_{\R \times \T^2} \La^{-1/2}(\th_n - f_n) \vphi dx dt 
	\right |
	&\leq
	\label{eq: h-prin int 1}
	\left |
		\int_{\R \times \T^2} \La^{-1/2}(\th_n - f_n) \tilde{\vphi} dx dt 
	\right |
	\\
	&+
	\label{eq: h-prin int 2}
	\left |
		\int_{\R \times \T^2} \La^{-1/2}(\th_n - f_n) (\vphi - \tilde{\vphi}) dx dt 
	\right |.
\end{align}
Let $\Theta_{n(k)}$ be the corrections defined so that $\th_n - f_n = \sum_{k\geq 0}\Theta_{n(k)}$ and the integral \eqref{eq: h-prin int 1} is 
\[
	 \eqref{eq: h-prin int 1}
	 =
	\left |
		\int_{\R \times \T^2} \sum_{k \geq 0} \La^{-1/2}\Theta_{n(k)} \tilde{\vphi} dx dt 
	\right |.	 
\]
From the Main Lemma \ref{lem: main}, $\La^{-1/2}\Theta_{n(k)} = \nb_i W^i_{n(k)}$, $\co{W_{n(k)}} \leq \hc \left( N \Xi_{n(k)} \right)^{-1}\D_{R, n(k)}^{1/2}$, and 
\ALI{
	\eqref{eq: h-prin int 1} 
	&\leq 
	\sum_{k \geq 0}
	\left |
		\int_{\R \times \T^2} \nb_i W^i_{n(k)} \tilde{\vphi} dx dt 
	\right |
	\leq 
	\sum_{k \geq 0}
	\co{ W_{n(k)}} \| \nb \tilde{\vphi} \|_{L^1}
	\leq
	\| \nb \tilde{\vphi} \|_{L^1}
	\sum_{k \geq 0}
	\hc \left( N \Xi_{n(k)} \right )^{-1}\D_{R, n(k)}^{1/2}
	\\
	&\leq
	\frac{2 \hc \D_{R,n(0)}^{1/2}}{N \Xi_{n(0)}} \| \nb \tilde{\vphi} \|_{L^1}
	\leq
	\frac{\hc \dmd{\D}_R^{1/2}}{\Xi_{n(0)}} \| \nb \tilde{\vphi} \|_{L^1}
	< \ep.  
}
Here we used that $Z \gg 1$ and $\Xi_{n(k+1)}^{-1} \D_{R,n(k+1)}^{1/2} \leq Z^{-\fr{5}{3}-\fr{1}{2}}\Xi_{n(k)}^{-1} \D_{R,n(k)}^{1/2}$ to sum the geometric series, and the last inequality holds by \eqref{eq: initial freq condition} for all $n\geq M$ sufficiently large. To bound \eqref{eq: h-prin int 2}, we now observe 
\[
	\eqref{eq: h-prin int 2} 
	\leq 
	\co{ \La^{-1/2}( \th_n - f_n) }
	\| \vphi - \tilde{\vphi}\|_{L^1}.
\]
Next we show a bound on $\co{\La^{-1/2}( \th_n - f_n)}$. To start
$
	\co{ \La^{-1/2}( \th_n - f_n) }
	\leq
	\sum_{k\geq 0}
	\co{\La^{-1/2}\Theta_{n(k)}}
$
Recall from the Main Lemma \ref{lem: main} that $\co{\La^{-1/2}\Theta_{n(k)}} \leq \hc \D_{R, n(k)}^{1/2}$ and we use the ansatz from \eqref{eq: ansatz} to get $\D_{R, n(k)} = \frac{\D_{R,n(0)}}{Z^{k}}$. From this equation, it follows that 
\[
	\co{ \La^{-1/2}( \th_n - f_n) }
	\leq 
	\sum_{k\geq 0}\hc \D_{R, n(0)}^{1/2} Z^{-k/2}
	\leq
	2 \hc \D_{R, n(0)}^{1/2}
	=
	2 \hc \dmd{\D}_R^{1/2},
\]
and therefore 
\[
	\eqref{eq: h-prin int 2} 
	\leq
	2 \hc \dmd{\D}_R^{1/2}
	\| \vphi - \tilde{\vphi}\|_{L^1}
	< \ep.
\]
The last inequality is due to our choice of $\tilde{\vphi}$ in \eqref{eq: choice of vphi approx}. By combining the estimates of \eqref{eq:cptFreqSuppApprox}, \eqref{eq: h-prin int 1} and \eqref{eq: h-prin int 2}, we have now shown that there exists an $M$ such that for all $n\geq M$
\[	
	\left |
		\int_{\R \times \T^2} \La^{-1/2}(\th_n - f) \vphi dx dt 
	\right |
	< 3\ep,
\]
and we can conclude the claim that $\La^{-1/2}\th_n \rightharpoonup \La^{-1/2}f$ in the $L^\infty$ weak-$\ast$ topology. 
\bibliographystyle{abbrv}
\bibliography{eulerOnRn}

\end{document}